\documentclass[a4paper,10pt]{amsart}
\usepackage[utf8]{inputenc}
\usepackage{amsthm,amsmath,amssymb,mathtools,mathrsfs,setspace}
\usepackage{bm}
\usepackage{graphicx}
\usepackage{todonotes}
\usepackage{ esint }
\usepackage{thmtools} 

\usepackage[pdfdisplaydoctitle,colorlinks,breaklinks,urlcolor=blue,linkcolor=blue,citecolor=blue]{hyperref} 

\newcommand{\C}{\mathbb{C}}

\newcommand{\N}{\mathbb{N}}

\newcommand{\R}{\mathbb{R}}

\newcommand{\U}{\mathcal{U}}

\renewcommand{\epsilon}{\varepsilon}

\newcommand{\brak}[1]{\left\langle#1\right\rangle}

\newtheorem{theorem}{Theorem}[section]
\newtheorem{definition}[theorem]{Definition}
\newtheorem{corollary}[theorem]{Corollary}
\newtheorem{lemma}[theorem]{Lemma}
\newtheorem{proposition}[theorem]{Proposition}

\theoremstyle{remark}

\newtheorem{Nota}[theorem]{Nota}

\numberwithin{equation}{section}

\usepackage{biblatex} 
\providecommand{\U}[1]{\protect\rule{.1in}{.1in}}

\begin{document}

\title{Coagulation dynamics under random field: turbulence effect on rain}



\author[A. Papini]{Andrea Papini}
\address{Scuola Normale Superiore, Piazza dei Cavalieri, 7, 56126 Pisa, Italia}
\email{\href{mailto:andrea.papini@sns.it}{andrea.papini@sns.it}}

\date\today

\begin{abstract}
    Turbulence in growth of rain droplet and rain formation is studied under an approximating particle system representing aggregation at the level of individuals, depending on their volume and distance in space, of the Smoluchoski Coagulation equation. A random field is introduced to model the air flow interaction with the particle and it is proved that the empirical density of the individual converges to solution of limiting equation under different setting for the random field of interaction. A brief numerical study for the continuous density is proposed using the particles approach, to analyze how noise can arise in such system and the effectiveness on rain formation.
\end{abstract}

\maketitle
\tableofcontents

\section{Introduction}
Although clouds play a crucial role in atmospheric phenomena, the effect of turbulence in cloud formation is not well established up to the present.\\
Two processes mutually affecting each other characterize the cloud formation and precipitation development: macro-scale processes such as the fluid motion of air associated with
clouds and the micro-scale processes such as condensation, stochastic
coalescence, and evaporation of water droplets. Thus we can say that cloud formation and precipitation development are typical multiscale-multiphysics phenomena.\\
Motivated by the study of \cite{SuperDrop}, in which the authors propose a novel algorithm to compute the stochastic coalescence of large number of droplet; and \cite{EulLag}, in which the problem of the role of turbulence in the cloud droplet growth is treated with simulation of density equation; we investigate the time of first formation of a raindrop, for the spatial and volume density in a cloud formation, under different settings, to explore the role of turbulent flow in the aggregation of small droplets in clouds.\\
\\
In doing so, we consider the random Smoluchowski coagulation equation for the spatial-volumetric density of the atmosphere
\begin{align}\label{smoleq}
    &\partial_t \rho_t(x,v) + \sum_{j\in H}\left(\sigma_j(x,v)\cdot \nabla_x\rho_t(x,v)\right)\xi^j_t=\varepsilon\Delta_x\rho_t(x,v)+F_v(\rho_t(x,\cdot))\\
    &F_v(\rho_t(x,\cdot)):=\frac{1}{2}\int_0^v\rho_t(x,w)\rho_t(x,v-w)f(w,v-w) dw-\int_{\R_+}\rho_t(x,v)\rho_t(x,w)f(v,w) dw\nonumber
\end{align}
Although Coagulation equations with diffusion are already being studied \cite{aldous1} both in discrete (\cite{hamsurv},\cite{ball1990discrete}) and continuous setting (\cite{hamsurv},\cite{1937-1632_2020_12_3319}), as limit of particle system (\cite{YAGHOUTI20093042},\cite{2006JSP...124..997H,},\cite{ines}) or as regularity of its solution \cite{AIHPC_2010__27_2_639_0}. Little is know about the counterpart $(\ref{smoleq})$ which capture in detail the more complex motion of rain droplet in the atmosphere and is object of this study. Here a diffusion term and a random field act on the cloud density taking account of the macro-scale process, such as fluid motion, acting on the formation.\\
The airflow in which they interact, solution of Navier-Stokes equation, is thus replaced with a random field with finite correlation time, depending on an environmental random noise
produced by Ornstein-Uhlenbeck processes.\\
This choice stems from the idea of Stochastic Model
Reduction \cite{maj} and can be replaced in future work with different approximations.\\
While the non linear term $F_v(\rho_t(x,\cdot))$ is the classic coagulation reaction interaction, as used in literature (\cite{10.1214/aoap/1029962598},\cite{aldous1},\cite{amann}).\\
\\
Objective of this paper is the understanding of quantities like
\begin{align}\label{tempcont}
    &\tau_f(\rho):=\inf\{t\geq 0 | \exists\ \eta >0,\ x\in D:\ \rho_t(x,R_{rd})\geq\eta\}
\end{align}
which represent the first time the density put enough mass on effective rain drops.\\
Our goal is to find numerical insight on the decreasing behaviour of such quantity in dependence of the parameter of the system.\\
To this end, we exchange density solution $\rho_t$ of $eq.$\ref{smoleq}, with an approximating particle system in which the aforementioned quantities $\tau_f(\rho)$ is replaced by
\begin{align}
    &\tau_f^N:=\inf\{t\geq 0 | \exists i\in N:\  R^i_t\geq R_{rd}\},
\end{align}
where $N$ represent the particles in the system and $R_i$, $R_{rd}$ are the radius of a particle and a real rain drop respectively.\\
\\
With this in mind, in Section.\ref{psm}, we rigorously define a particle system dynamics that, in addition to being easier to simulate, approximate  our key Smoluchoski equation \ref{smoleq}.\\
Warm clouds and rain consist of large number of water droplets with broad range of radius from micrometers to several millimeters. Several sorts of aerosol particles are also floating in the atmosphere, which can be a seed of a cloud droplet.\\
Let $N$ be the number of such individuals (rain droplets and dust particles) in the considered region, each particle is characterized by its position and volume: $\left(X_t^{i,N},V_t^{i,N}\right),\ i=1, .., N$.\\
To describe the dynamics we see how particle acts on observable $F$ of the system, using the infinitesimal generator $\mathcal{L}$ of the dynamical process.
$$
\mathcal{L}F(\eta)=\mathcal{L}_D F(\eta)+\mathcal{L}_C F(\eta).
$$
The first term on the r.h.s. represents the diffusion and random field acting on the particles
$$
\mathcal{L}_D F(\eta):=\sum_i \varepsilon\Delta_{x_i}F(\eta)+\sum_i \mathcal{U}_t(x_i)\cdot\nabla_{x_i}F(\eta).
$$
While the second term, representing the coalescence of the particle, is a mean field counterpart of the classic coagulation interaction term:
\begin{align*}
&\mathcal{L}_C F(\eta):=\sum_{i,j}T_N(i,j)\left[\mathbb{I}_{i,j}F(\eta^{i,j})+\mathbb{I}_{j,i}F(\eta^{j,i})-F(\eta)\right]\\
&\mathbb{I}_{i,j}:=\mathbb{I}_{(v_i>=v_j)},\ \mathbb{I}_{j,i}:=\mathbb{I}_{(v_j>v_i)};\\
&T_N^\delta(i,j):=\frac{1}{2} N^{-1}\delta^{-3}\mathbb{I}_{B_{x_i}(\delta f(v_i,v_j))}(x_j)k(v_i,v_j)
\end{align*}
Defining $\mu_t(x,v):=\frac{1}{N}\sum_{i=1}^N \delta_{(x_t^i,v_t^i)}$ as the empirical measure of the system and coupling it with the random noise $\xi_t^h,\ h\in H$ we are able to obtain convergence $(Thm.\ 2.1)$ of the approximating particle dynamics to a solution of the non local Random Smoluchoski equation, parameterized by $\delta$
\begin{align}\label{dsmoleq}
    &\partial_t \rho^\delta_t(x,v) + \sum_{j\in H}\left(\sigma_j(x,v)\cdot \nabla_x\rho^\delta_t(x,v)\right)\xi^j_t=\varepsilon\Delta_x\rho^\delta_t(x,v)+F^\delta_v(\rho^\delta_t(x,\cdot))
\end{align}
where $F^\delta$ is the non local coagulation kernel closely related to equation \ref{smoleq} as we have shown in \ref{locint}.\\
\\
To this end, Section.\ref{tcseq} is devoted to a formal study of the mean field limit equation. In the section we show that, under suitable hypothesis, steamed from (\cite{article1},\cite{article2}), we are able to recover a solution of the coagulation-position system (\ref{smoleq}) as show in Proposition $3.1$ as a local convergence in the volume variable
\begin{align*}
    \sup_{t\in[0,T]}\|\rho_t-\rho_t^\delta\|_{L_x^2(\R^d)L^1_{v}(0,R)}\rightarrow 0,\ as\ \delta\rightarrow0,\ \forall\ R>0, \mathbb{P}-a.s.
\end{align*}
Justifying the use of the approximating particle system to compute meaningful quantities of the system of interest and recover information on the effect of the random field on coagulation of droplet in the the region  of space considered.\\
\\
As such, in Section.\ref{numan}, numerical regressions are performed on the particle system to study the mean time formation of rain droplet to recover a decay dependence of $\tau_f^N$ in dependence of the parameter of the random field $|H|$ and $\lambda$: the number of vortex and the intensity of the environmental noise.\\
In particular to analyze how an independent noise and a random turbulent vortex flow could affect the time of rain formation in such a system and the effectiveness of the formation.\\
We studied two different settings: first, as a benchmark, we show numerical results concerning the simulation of
particles subjected only to independent Brownian Motion. We recover the classical conjecture such as the quadratic decay of the time formation with respect to the intensity of the independent noise.\\
The second setting considers the particles under the turbulent random flow and we find an analogous decay of the time of first rain formation, yet with a less efficient rate than the independent noise.\\
In particular, the numerical results show that is expected to have a behavior of such exit time of the form
$$
\mathbb{E}[\tau_f^N]\sim p(|H|,\lambda)^{-1},
$$
with $p(\cdot,\cdot)$ polynomial of at most degree two; showing that although with a less efficient rate that independent theoretical Brownian motion, turbulence produced by the motion of the droplets in the atmosphere under the wind flow enhance the coagulation and the rain formation.
\section{Particle System Model}\label{psm}
The microscopic model that we study in this article consists of a large number of particles, in the space $\R^d,\ d\geq 3$, which move according to independent Brownian motions and a random field acting as the wind on the particle.
\begin{align*}
&dX_t^i=\mathcal{U}_t(X^i_t,\omega)dt+\varepsilon dB^i_t\\
&\mathcal{U}_t(x):=\sum_{k\in K}\sigma_k(x)\xi^k_t\\
&d\xi_t=-\lambda\xi_t dt+\lambda dW^i_t,\ \xi_0=0.
\end{align*}
This velocity field $\mathcal{U}:=\mathcal{U}^{\lambda}, \lambda>1$, with finite correlation time, is of the following form:
\[
\mathcal{U}^{\lambda}\left(  t,x\right)  =\sum_{h\in H}\sigma_{h}\left(  x\right)  \xi_{t}^{h,\lambda}%
\]%
where $H$ is a set of index with finite cardinality, $\{\sigma_h\}_h$ is a suitable family of time independent vector field and $\{\xi_t^{h,\lambda}\}_h$ is a family of stationary i.i.d. Ornstein-Uhlenbeck processes with covariance $cov(\xi_t^{h,\lambda},\xi_s^{h,\lambda})=\frac{\lambda}{2}exp(\lambda |t-s|)$, solution of
\[
d\xi_{t}^{h}  =-\lambda\xi_{t}^{h}dt+\sigma dW_{t}^{h},\ \ \  
\xi_{0}^{h}  =0,\ k  =1,...,K.\\
\]%
For a family of i.i.d. Brownian motions $\{W_t^h\}_{h\in H}$ on a probability space $\left(\Omega, \{\mathcal{F}_t\}_{t\in[0,T]}, \mathbb{P}\right)$. Up to changing the filtration, we also have an explicit formula for the solution of such object
$$
\xi_t^{h,\lambda}=\lambda\int_0^t\ e^{-\lambda(t-s)}dW_s^h,\ t\in[0,T]
$$
During our work we select smooth vector fields $\{\sigma_h\}_h$, ideally as smooth approximation of the Biot-Savart Kernel, such that fixing point $x_h\in\R^2$ is defined as
\[
\sigma_{h}\left(  x\right)  \sim\frac{1}{2\pi}\frac{\left(  x-x_{h}\right)
^{\perp}}{\left\vert x-x_{h}\right\vert ^{2}},\ \left(  a,b\right)  ^{\perp}:=\left(  b,-a\right)%
\]%
This choice, admittedly phenomenological, come from the idea that one would like to take
$\mathcal{U}_t(x)\equiv U(x)$, the solution of the Navier-Stokes equations for the wind velocity, or a slight modification, and the use of this "vortex like structure" was the proposed approximation.\\
\\
To each particle we attach a value $v\in (0,\infty)$ which represent the volume of the particle (we assume that mass and volume are around the same and can be recovered from one and another). In the mean field model any pair of particles that approach to within a certain range of interaction are liable to coagulate, at which time they disappear from the system, to be replaced by a particle whose volume is equal to the sum of the volumes of the colliding particles, and whose location is a specific point in the vicinity of the location of the coagulation.\\
\\
As a matter of convenience, we introduce the microscopic model, where the number of particles is initially deterministic. We define a sequence of microscopic models,indexed by a positive integer $N$. A countable set $I:=\N,...$ of symbols is provided. A configuration $\eta$ is an $\R^d\times(0,\infty)-valued$ function on a finite subset $I_\eta\subset I$. For any $i\in I_\eta$, the component $\eta(i)$ may be written as $(x_i,v_i)$. The particle labelled by $i$ has volume $v_i$ and location $x_i$.\\
To describe this dynamics we need to see how the particle acts on observable of the system.\\ Let $F:(\R^d\times\R^0_+)^N\rightarrow\R$ denote a function such that $F:(\R^d\times\R_+)^N\rightarrow\R$ is smooth, and $F(x,0)=0,\forall x\in\R^d$. The action on $F$ of the infinitesimal generator $\mathcal{L}$ is given by
$$
\mathcal{L}F(\eta)=\mathcal{L}_D F(\eta)+\mathcal{L}_C F(\eta).
$$
The first term on the r.h.s. represents the diffusion and random field acting on the particles
$$
\mathcal{L}_D F(\eta):=\sum_i \varepsilon\Delta_{x_i}F(\eta)+\sum_i \mathcal{U}_t(x_i)\cdot\nabla_{x_i}F(\eta).
$$
While for the second term, representing the coalescence of the particle we have:
\begin{align*}
&\mathcal{L}_C F(\eta):=\sum_{i,j}T_N(i,j)\left[\mathbb{I}_{i,j}F(\eta^{i,j})+\mathbb{I}_{j,i}F(\eta^{j,i})-F(\eta)\right]\\
&\mathbb{I}_{i,j}:=\mathbb{I}_{(v_i>=v_j)},\ \mathbb{I}_{j,i}:=\mathbb{I}_{(v_j>v_i)};\\
&T_N^\delta(i,j):=\frac{1}{2} N^{-1}\delta^{-3}\mathbb{I}_{B_{x_i}(\delta f(v_i,v_j))}(x_j)k(v_i,v_j)
\end{align*}
here we have indicated 
$$
\eta^{i,j}_k=\begin{cases} 
\eta(k), & \textit{if}\ k\neq i,j\\
(x_i,v_i+v_j), & \textit{if}\ k=i\\
(x_i,0), & \textit{if}\ k=j.
\end{cases}
$$
We note here, as one can see from $\mathcal{L}_C$, that the coalescence generator is as such that two particles interact if they are close in the space depending not only on the radius, but also on some fixed scaling quantities $\delta$.\\
This is done to obtain a scaling limit in the mean field regime for the system. As such, we needed to scale the $T$ interaction Kernel with a factor $1/N$, that take into account the interaction with the particle as $N$ grows.\\
In particular, as $N$ grows the particle are nearer, with mean distance $N^{-1/d}$, and so when we consider the $T$ function we should make a rescale in the position: two particles interact if $|X_t^{i,N}-X^{j,N}_t|\leq N^{-1/d}$, and since in our case the interaction depends also on the volume of each raindrop, we would like to have
$$
|X_t^{i,N}-X^{j,N}_t|\leq N^{-1/d}f(V_t^{i,N},V^{j,N}_t))
$$
and thus making it a local problem. This is out of our capacity for now, so we select a rescaling parameter $\delta$ such that $T$ is substituted whit $T^\delta$ of the form
$$
\delta^{\alpha}\mathbb{I}_{B_x(\delta f(v,w)}g(v,w)
$$
where $\alpha$ depend on the system (in particular the dimension) and $\delta$ depends on $N$.\\
\\
We note that $(x_i(t),v_i(t))\in C([0,T],R^d)\times \mathcal{D}([0,T],R_+)$, we consider now for each $N$ the empirical measure associated to the particles
$$
\mu_t^N:=\frac{1}{N}\sum_{i=1}^N \delta_{(X_i^N(t),V_i^N(t))}\in \mathcal{D}([0,T],\mathcal{M}_1(\R^d\times\R_+)
$$
As in the other setting we consider the pair $(\mu_t^N,\xi_t)\in \mathcal{D}([0,T],\mathcal{M}_1(\R^d\times\R_+))\times C([0,T],\R)$. If we consider now $\phi\in C^{1,2}_b$, $\phi(x,0)=0$, non necessarily continuous in $v=0$. we can compute the infinitesimal generator acting on $\brak{\mu_t^N,\phi}$ and we have
\begin{align*}
    \brak{\mu_t^N,\phi}=&\brak{\mu_0^N,\phi}+\int_0^t\varepsilon\brak{\mu_s^N,\Delta\phi}ds+\int_0^t \brak{\mu_s^N,\mathcal{U}_s\cdot\nabla\phi}ds +\\
    &+\int_0^t \brak{\mu_t^N,\brak{\mu_t^N,T^\delta(x,v,y,w)J^\phi(x,v,y,w)}}ds+\tilde{M}_t
\end{align*}
where
\begin{align*}
&T^\delta(x,v,y,w):=\frac{1}{2}\delta^{-3}\mathbb{I}_{B_{x}(\delta f(v,w))}(y)k(v,w);\\
&J^\phi(x,v,y,w):=\left[\mathbb{I}_{(v>=w)}\phi(x,v+w)+\mathbb{I}_{v<w}\phi(y,v+w)-\phi(x,v)-\phi(y,w)\right]
\end{align*}
$\tilde{M}_t$ is a martingale and we can compute is quadratic variation:
$$
\mathbb{E}\left[\tilde{M}_t^2\right]=\mathbb{E}\left[\int_0^t\left(\mathcal{L}\brak{\mu_s^N,\phi}^2-2\brak{\mu_s^N,\phi}\mathcal{L}\brak{\mu_s^N,\phi}ds \right)\right]
$$
which we can compute in two parts $M_D,\ M_C$ such that:
\begin{align*}
& \brak{M_D}:=\mathbb{E}\left[\int_0^t N^{-1}\brak{\mu_t^N,|\nabla\phi|^2}ds\right]\\
& \brak{M_C}:=\mathbb{E}\left[\int_0^t \frac{1}{N^3}\sum_{i,j}T_t^{\delta}(i,j)J_t^{\phi}(i,j)^2ds\right]=\mathbb{E}\left[\int_0^t N^{-1}\brak{\mu_s^N,\brak{\mu_s^N, T^\delta(x,v,y,w)J^\phi(x,v,y,w)^2}}ds\right]
\end{align*}
The last computation is a very tedious one, but not difficult.\\

Our result is the following
\begin{theorem}
Assuming that $(\mu,\xi)$ realise a limiting law of the sequence $(\mu^N,\xi)$. Then $\mu$ is a continuous process taking values in $\mathcal{M}_1(\R^{d}\times\R_+)$ and satisfies the limit random PDE:
\begin{align*}
   \brak{\mu_t,\phi}-\brak{\mu_0,\phi}=&\int_0^t\ \brak{\mu_s,\frac{\varepsilon_x^2}{2}\Delta_x\phi}ds+ \int_0^t\ \sum_{j\in J}\brak{\mu_s,\sigma_j(x,v)\cdot\nabla_x\phi}\xi^j_s ds \\
   &+\int_0^t\ \brak{\mu_s,\brak{\mu_s,T^\delta J^\phi}_{y,w}}_{x,v}ds
\end{align*}
for every $t\in[0,T]$ and $\phi\in C_{test}$, with probability 1. Here:
\begin{align*}
&T^\delta(x,v,y,w):=\frac{1}{2}\delta^{-3}\mathbb{I}_{B_{x}(\delta f(v,w))}(y)k(v,w);\\
&J^\phi(x,v,y,w):=\left[\mathbb{I}_{(v>=w)}\phi(x,v+w)+\mathbb{I}_{v<w}\phi(y,v+w)-\phi(x,v)-\phi(y,w)\right].
\end{align*}
Moreover, if $\mu_0$ has finite volume of the particle, also $\mu_t$ does and we have:
$$
\sup_{t\leq T}\int_{\R^d\times\R_+}v\mu_t(dx,dv)\leq \int_{\R^d\times\R_+}v\mu_0(dx,dv)
$$
\end{theorem}
Later we'll show under suitable hypothesis how to pass to the limit as $\delta\rightarrow0$ and retrieve that $\mu_t^\delta\rightarrow\mu_t$ where $\mu_t$ is a solution in the sense of weak formulation for the Random Smoluchoski equation, justifying our particle system approximation scheme:
\begin{definition}
A continuous family of measure $\mu_t\in C([0,T],\mathcal{M}_1(\R^d\times\R)$, given $\xi_t$ as before, is a weak solution of the RSPDE if
\begin{align*}
   \brak{\mu_t,\phi}-\brak{\mu_0,\phi}=&\int_0^t\ \brak{\mu_s,\frac{\varepsilon_x^2}{2}\Delta_x\phi}ds+ \int_0^t\ \sum_{j\in J}\brak{\mu_s,\sigma_j(x,v)\cdot\nabla_x\phi}\xi^j_s ds \\
   &+\int_0^t\ \brak{\mu_s,\brak{\mu_s,k(v,w)T(v,w) J_{x,v,w}^\phi}_{x,w}}_{x,v}ds
\end{align*}
for every $t\in[0,T]$ and $\phi\in C_{test}$, with probability 1. Here:
\begin{align*}
&K(v,w)T(v,w):=\frac{\pi}{2}\left(v^{1/3}+w^{1/3}\right)^3E(v,w),\ with\ E\in C^0_b;\\
&J^\phi(x,v,w):=\left[\phi(x,v+w)-\phi(x,v)-\phi(x,w)\right].
\end{align*}
\end{definition}

\subsection{Scaling Limit: Uniform Estimates and Compactness}
Since $\xi$ does not change in the sequence $(\mu^N,\xi)$, to prove tightness for law $Q^N:=\mathcal{L}(\mu^N,\xi)$ we need to prove it just for the marginal $\mu^N_t$, which is done in the following computation.\\
Suppose that the initial condition
$(X_0^{i,N},V_0^{i,N})$ of the system  are i.i.d from a probability distribution with law $\mu_0$ and $\mathcal{F}_0-measurable$, suppose also that:\\
\\
$a)$ $\sup_{i,N}\mathbb{E}\left[|X_0^{i,N}|+|V_0^{i,N}|\right]<\infty$;\\
$b)$ $\exists\mu_0\in Pr_1(\R^d\times\R_+)$ such that $\brak{\mu_0^N,\phi}\rightarrow_{\mathbb{P}}\brak{\mu_0,\phi}$, for every $\phi\in C^\infty(\R^d\times\R_+)$.\\
$c)$ $\int_{\R^d}v\mu_0(dx,dv)<\infty$;\\
\\
This is true under the particular assumption made in the introduction: $(X_0^{i,N},V_0^{i,N})=(X_0^i,V_0^{i})$, where $\{X_0^i,V_0^i\}_{i\in\N}$ is a sequence of i.i.d. $\mathcal{F}_0-$measurable r.v.'s with common law $\mu_0\in Pr_1(\R^d\times\R_+)$.\\
We also assume the following on system $(3.1)$:\\
\\
$c)$ $\mathcal{U}_t(x,v,\omega)$ is of the form $\sum_{h\in H}\sigma_h(x)D_h(v)\xi_t^h(\omega),\ |H|<\infty$, where $\sigma_h$ is regular  $\forall h, \|\sigma_h\|_\infty\leq K_\sigma<\infty$, and $D_h$ is smooth and bounded with a constant $K_D$;\\
$d)$ the interacting kernel for the volume of the particle is defined as $$T^\delta:=\delta^{-3}\mathbb{I}_{B_x(\delta f(v,w)}(y) k(v,w).$$ Here $k$ represents the efficiency collision and is such that $k\in C^0_b,\ \|k\|_\infty<\infty$, and $f$ is as usual.\\ We'll derive a real interaction kernel at the limit of $\delta\rightarrow0$ later in the next chapter.\\
\\
The crucial step
in the proof of tightness is to establish suitable uniform bounds for the increments of the process $\mu_t^N$.
\begin{proposition}
Let $T\geq 0$ and $\phi\in C^1,2$ be fixed. There exists a constant
$C> 0$ depending only on $T,\phi,\xi,K_\sigma,\delta$ such that for all $N\in\N$, any bounded stopping time
$\tau\leq T$ with respect to the filtration of $\mu^N$, and $\vartheta>0$, it holds
$$
\mathbb{E}\left[|\brak{\mu^N_{\tau+\vartheta}-\mu^N_\tau,\phi}|^2\right]\leq C\vartheta.
$$
\end{proposition}
The proof of the latter shall be divided in two parts, corresponding to the generator part and the stochastic (martingale) part of the dynamics of $\mu^N$. Indeed, we can always expand the above increments as
\begin{align*}
    \brak{\mu^N_{\tau+\vartheta}-\mu^N_\tau,\phi}=&\int_\tau^{\tau+\vartheta}\mathcal{L}_s\brak{\mu^N_s,\phi}ds+\left(\tilde{M}_{\tau+\vartheta}^{\phi,N}-\tilde{M}^{N,\phi}_\tau\right).
\end{align*}
where $\tilde{M}$ is a martingale and we know the quadratic variation. Thus we can reduce ourselves to bound separately the first and second summands
in the right-hand side of the expression above.
\begin{lemma}
In the notation above, it holds
$$
\mathbb{E}\left[\big|\int_\tau^{\tau+\vartheta}\mathcal{L}\brak{\mu_t^N,\phi}dt\big|^2\right]\leq C\vartheta
$$
where the constant $C> 0$ only depends on $\xi,\phi,T,K_\sigma$.
\end{lemma}
\begin{proof}
Consider the integrating function $\mathcal{L}\brak{\mu_s^N,\phi}$ we can explicitly bound
\begin{align*}
    |\mathcal{L}\brak{\mu_s^N,\phi}|&\leq |\mathcal{L}_D\brak{\mu_s^N,\phi}|+|\mathcal{L}_C\brak{\mu_s^N,\phi}|\leq|\brak{\mu_s^N,\frac{\varepsilon^2}{2}\Delta\phi}|+|\brak{\mu_s^N,\xi_s\sigma\cdot\nabla_x\phi}|+\\
    &+|\brak{\mu_s^N,\brak{\mu_s^N,T^\delta J^\phi}}|\\
    &\leq \frac{\varepsilon^2}{2}\|\phi\|_{1,2}+K_\sigma \sup_{t\in[0,T]}|\xi_t|\|\phi\|_{1,2}+2\delta^{-3}\|k\|_{\infty}\|\phi\|_{\infty}.
\end{align*}
Passing to the square, using Jensen inequality we get
$$
\mathbb{E}\left[\big|\int_\tau^{\tau+\vartheta}\mathcal{L}\brak{\mu_t^N,\phi}dt\big|^2\right]\leq C^2_{\xi,\phi,\sigma,k,\delta}\ \vartheta.
$$
\end{proof}
\begin{lemma}
In the notation above, it holds
$$
\mathbb{E}\left[\left(\tilde{M}^N_{\tau+\vartheta}-\tilde{M}^N_\tau\right)^2\right]\leq \frac{C\vartheta}{N},
$$
where the constant $C>0$ only depends on $\phi,T,\delta$.
\end{lemma}
\begin{proof}
Using the fact that $\tau$ is a stopping time for $\mu^N$ and $\tilde{M}$ is a martingale, follows
\begin{align*}
\mathbb{E}\left[\left(\tilde{M}^N_{\tau+\vartheta}-\tilde{M}^N_\tau\right)^2\right]&\leq\mathbb{E}\left[\big|\int_\tau^{\tau+\vartheta}\mathcal{L}\brak{\mu_t^N,\phi}^2-2\brak{\mu_t^N,\phi}\mathcal{L}\brak{\mu_t^N,\phi}dt\big|\right]\\
&\leq \vartheta\mathbb{E}\left[\sup_t\left( \mathcal{L}\brak{\mu_t^N,\phi}^2-2\brak{\mu_t^N,\phi}\mathcal{L}\brak{\mu_t^N,\phi}\right)\right].
\end{align*}
From the definition of the two martingale (the diffusive one and coagulation one) we have:
\begin{align*}
& \brak{M_D}:=\mathbb{E}\left[ N^{-1}\brak{\mu_t^N,|\nabla\phi|^2}\right]\leq N^{-1} \|\phi\|_{1,2}\\
& \brak{M_C}:=\mathbb{E}\left[\frac{1}{N^3}\sum_{i,j}T_t^{\delta}(i,j)J_t^{\phi}(i,j)^2\right]\leq N^{-1}2\|\phi\|_{\infty}\delta^{-3}\|k\|_{\infty}.
\end{align*}
Which concludes the proof.
\end{proof}
The combination of Lemma 2.4 and Lemma 2.5 proves Proposition 2.3, from
which we immediately deduce the following tightness result
\begin{proposition}
Under the hypothesis of the previous lemmas, $\{Q^N\}_{N\in\N}$ the law of $\{\mu^N_\cdot\}_{N\in\N}$ on $\mathcal{D}([0,T],\mathcal{M}_1(\R^d\times\R_+))$ is relatively compact.\\
As a consequence the law of $(\mu^N,\xi)\in \mathcal{D}([0,T],\mathcal{M}_1(\R^d\times\R_+))\times C([0,T],\R)$ is tight and there exist a sub sequence such that $(\mu^N_k,\xi)\rightarrow_{\mathcal{L}}Q$ a law in $\mathcal{D}([0,T],\mathcal{M}_1(\R^d\times\R_+))\times C([0,T],\R)$.
\end{proposition}
\begin{proof}

it suffices to show that the laws of couplings$\brak{\mu_t^N,\phi}$ is tight on $\mathcal{D}([0,T],\R)$ for any fixed $\phi$ in a dense subset of $C^0$
including the constant $\phi \equiv 1$. Such tightness is in turn verified if:
\begin{itemize}
    \item for any $t\in[0,T]$ the sequence of random variables $\brak{\mu_t^N,\phi}$ is tight;
    \item for any $\overline{\delta} > 0$ it holds
    $$
    \lim_{\varepsilon\rightarrow0}\limsup_{N\in\N}\sup_{\tau\in \mathcal{T}_T^N}\sup_{\vartheta\leq \varepsilon} Q^N(|\brak{\mu^N_{\tau+\vartheta}-\mu^N_\tau,\phi}|>\delta)=0
    $$
    where $\mathcal{T}_T^N$ denotes the family of stopping times with respect to the filtration
of $\mu_t^N$ bounded by $T$, that is $\tau\leq T$ almost surely;
\end{itemize}
The former condition is an easy consequence since $|\mu_t^N|\leq 1$, the collection $(\brak{\phi,\mu_t^N})_{N\in\N}$ for fixed $t$ is automatically tight because $\brak{\phi,\mu_t^N}\leq \|\phi\|_\infty$. The second condition, follows immediately from  the previous proposition and Markov inequality.
\end{proof}
\subsection{Passing to the Limit}
As a consequence of proposition $(2.6)$ the law of $(\mu^N,\xi)\in C([0,T],Pr_1(\R^d\times\R_+))\times C([0,T],\R)$ is tight and we can extract a subsequence, whcich we still denote $Q^N$, such that converges weakly to a law $Q$ in $C([0,T],Pr_1(\R^d\times\R_+))\times C([0,T],\R)$.\\
Let us now consider $(\mu_t,\xi)$, a process such that its law is equal to $Q$ and such that
$$
(\mu_t^{N_k},\xi)\rightarrow_{\mathcal{L}}(\mu_t,\xi).
$$

\begin{theorem}
If $(\mu,\xi)$ realise the law of a weak limit point of a subsequence of $\{Q^N\}_{N\in\N}$, then $\mu$ solves the RPDE with probability 1 over test function.
\end{theorem}
\begin{proof} 
We consider the functional
\begin{align*}
\Psi_{\phi}(\nu_\cdot,f_\cdot):=&\sup_{t\in [0,T]}\big|\brak{\nu_t,\phi}-\brak{\nu_0,\phi}-\int_0^t\ \brak{\nu_s,\frac{\varepsilon_x^2}{2}\Delta_x\phi}ds\\
   &-\int_0^t\ \brak{\nu_s,\sigma(x,v)\cdot\nabla_x\phi}f_s\ ds -\int_0^t\ \brak{\nu_s,\brak{\nu_s,T^\delta J^\phi}}ds\ \big|
\end{align*}
defined for every $\phi\in C^{1,2}_b(\R^d\times\R)$.\\
It is continuous on $\mathcal{D}([0.T],Pr_1(\R^d\times\R))\times C([0,T],\R)$.\\
Hence, if $(\mu_t^{N_k},\xi)\rightarrow_{\mathcal{L}}(\mu_t,\xi)$ is a subsequence which weakly converges to $(\mu,\xi)$, by Portmanteau theorem we have
\begin{align*}
    \mathbb{P}(\Psi_{\phi}(\mu_\cdot,\xi_\cdot)>\delta)\leq \liminf_{k\rightarrow\infty}\mathbb{P}\left(\Psi_{\phi}(\mu^{N_k}_\cdot,\xi_\cdot)>\tilde{\delta}\right).
\end{align*} 
Using the definition of the functional and the identity satisfied by $(\mu_t^N,\xi)$ in $lemma\ (3.3)$, the r.h.s. is equal to
$$
=\mathbb{P}\left(\sup_{t\in[0,T]}|\brak{\mu_0^{N_k},\phi}-\brak{\mu_0,\phi}+M_t^{\phi,N_k,x}+M_t^{\phi,N_k,v}|>\tilde{\delta}\right).
$$
Using the hypothesis on the initial condition, regularity of the velocity field and interaction kernel, standard computation for the martingale part gives us
$$
    \mathbb{P}(\Psi_{\phi}(\mu_\cdot,\xi_\cdot)>\tilde{\delta})\leq\liminf_{k\rightarrow\infty}\mathbb{P}\left(\Psi_{\phi}(\mu^{N_k}_\cdot,\xi_\cdot)>\tilde{\delta}\right)=0.
$$
Since this hold true for every $\tilde{\delta}>0$, we deduce that
$$
\mathbb{P}(\Psi_{\phi}(\mu_\cdot,\xi_\cdot)>\tilde{\delta})=1.
$$
Hence $\forall\phi\in C^{1,2}_b(\R^d\times\R)$:
\begin{align*}
   \brak{\mu_t,\phi}-\brak{\mu_0,\phi}=&\int_0^t\ \brak{\mu_s,\frac{\varepsilon_x^2}{2}\Delta_x\phi}ds \\
   &+\int_0^t\ \brak{\mu_s,\sigma(x,v)\cdot\nabla_x\phi}\xi_s ds +\int_0^t\ \brak{\nu_s,\brak{\nu_s,T^\delta J^\phi}}ds
\end{align*}
$\forall t\in [0,T]$, with probability 1. 
\end{proof}
We now can conclude the proof of Theorem (2.1)
\begin{proof}
What we need to prove to conclude the theorem is that we can swap the quantifiers on $\phi\in \mathcal{D}(R^{d+1})$.\\ Since this space is separable we can work on a infinite countable dense subset and we have
$$
\mathbb{P}(\forall\phi\ \Psi_\phi(\mu,\xi)=0)=\mathbb{P}(\forall\phi_n\ \Psi_{\phi_n}(\mu,\xi)=0)=1.
$$
And so we have that:
\begin{align*}
   \brak{\mu_t,\phi}-\brak{\mu_0,\phi}=&\int_0^t\ \brak{\mu_s,\frac{\varepsilon_x^2}{2}\Delta_x\phi}ds \\
   &+\int_0^t\ \brak{\mu_s,\sigma(x,v)\cdot\nabla_x\phi}\xi_s ds +\int_0^t\ \brak{\mu_s,\brak{\mu_s,T^\delta J^\phi}}ds
\end{align*}
for every $t\in[0,T]$ and $\phi\in C_{test}$, with probability 1.\\
The continuity of $\mu$ follows from the fact that every element in the r.h.s. is continuous. 
\end{proof}
What we have now is that any process realizing the limiting law is a solution of the random evolution equation in the sense above. What we need to conclude now is:
\begin{enumerate}
    \item Prove uniqueness of solutions for this equation.
    \item Conclude all limiting laws agree, and hence that we have convergence in law.
\end{enumerate} 
\begin{theorem}
If $(\mu,\xi),\ (\nu,\xi)$ are two processes such that $\mu,\nu$ satisfy the weak formulation as in $Thm.\ 8.3$, $\mathbb{P}-a.s.$, then:
$$\|\mu_t-\nu_t\|=0,\ \forall t\in[0,T],\ \mathbb{P}-a.s.$$
\end{theorem}
\begin{proof}
Existence in the weak sense has been proved above. Let us prove uniqueness. To
shorten the expressions we set $\varepsilon_x=\varepsilon_v=1$ and $|H|=1$.\\
Let $\mu$ and $\nu$ be two solutions. Using as test
functions the heat kernel multiplied by a test function we have that both satisfy
\begin{align*}
    \brak{\mu_t-\nu_t,\phi}=&\int_0^t\ \brak{\mu_s-\nu_s,\sigma(x,v)\cdot\nabla_x(e^{(t-s)\Delta}\phi)}\xi_s ds\\
    &+\int_0^t\ \brak{\mu_s,\brak{\mu_s,T^\delta J^{(e^{(t-s)\Delta}\phi)}}}ds\\
    &-\int_0^t\ \brak{\nu_s,\brak{\nu_s,T^\delta J^{(e^{(t-s)\Delta}\phi)}}}ds
\end{align*}
We pass to the module and we estimate for $\mathbb{P}-a.s.$
\begin{align*}
    |\brak{\mu_t-\nu_t,\phi}|\leq&\int_0^t\ |\brak{\mu_s-\nu_s,\sigma(x,v)\cdot\nabla_x(e^{(t-s)\Delta}\phi)}|ds\|\xi\|_{L_t^\infty}\\
    &+\int_0^t\ |\brak{\mu_s-\nu_s,\brak{\mu_s,T^\delta J^{(e^{(t-s)\Delta}\phi)}}}|ds\\
    &+\int_0^t\ |\brak{\nu_s,\brak{\mu_s-\nu_s,T^\delta J^{(e^{(t-s)\Delta}\phi)}}}|ds
\end{align*}
We the first term in the r.h.s. we have $\mathbb{P}-a.s.$
\begin{align*}
\int_0^t\ |\brak{\mu_s-\nu_s,\sigma(x,v)\cdot\nabla_x(e^{(t-s)\Delta}\phi)}|ds\|\xi\|_{L_t^\infty}&\leq \int_0^tK_{\sigma} \|\phi\|_\infty\|\nabla e^{(t-s)\Delta}\|\|\mu_s-\nu_s\|ds \|\xi\|_{L_t^\infty}\\
&\leq\int_0^t\frac{C(\xi, \phi,\sigma)}{\sqrt{t-s}}\|\mu_s-\nu_s\|ds.
\end{align*}
For the second term in the r.h.s we have
\begin{align*}
\int_0^t\ |\brak{\mu_s-\nu_s,\brak{\mu_s,T^\delta J^{(e^{(t-s)\Delta}\phi)}}}|ds&\leq \int_0^t\ \brak{\mu_s-\nu_s,2\delta^{-3}\|\phi\|_{\infty}\|K\|_{\infty}}ds
\end{align*}
where we have used the fact that
\begin{align*}
|T^\delta J^{e^{(t-s)\Delta}\phi}|&=|\int \frac{1}{2}\delta^{-3}\mathbb{I}_{B_x(\delta f(v,w))}k(v,w)4|e^{(t-s)\Delta}\phi|d\mu^N_s|\\
&\leq 2\delta^{-3}\int d\mu_s^N(dx,dv)\|\phi\|_\infty\|k\|_{\infty}\leq 2\delta^{-3}\|\phi\|_{\infty}\|K\|_{\infty}<\infty.
\end{align*}
putting all together we have
$$
\int_0^t\|\brak{\mu_s-\nu_s,\brak{\mu_s,T^\delta J^{(e^{(t-s)\Delta}\phi)}}}|ds\leq\int_0^t C(\delta,k,\phi)\|\mu_s-\nu_s\|ds.
$$
For the last term we have
\begin{align*}
\int_0^t\ |\brak{\nu_s,\brak{\mu_s-\nu_s,T^\delta J^{(e^{(t-s)\Delta}\phi)}}}|ds&\leq C(k,\delta,\phi)\int_0^t \|\mu_s-\nu_s\|ds
\end{align*}
If we combine all the inequality we get, over the sup on the test function with $\|\phi\|=1$
$$
\|\mu_t-\nu_t\|\leq C(\xi,k,\sigma,T,\delta)\int_0^t \frac{\sqrt{t-s}+1}{\sqrt{t-s}}\|\mu_s-\nu_s\|ds
$$
and with a modification
of a standard lemma we are able to conclude that\\
$\|\mu_t-\nu_t\|=0.$
\end{proof}
\begin{corollary}
The family $\{Q^N\}_{N\in\N}$
of laws of the processes $(\mu_\cdot^N,\xi_\cdot)$ weakly converge
to $Q$ where $\mu\in C([0,T],\mathcal{M}_1(\R^d\times\R_+))$.
\end{corollary}
\begin{proof}[Proof.]
What we have now is that every subsequence of $\{Q^N\}_{N\in\N}$ admit a limit measure $Q$ such that every $(\mu,\xi)$ realizing the law solves $\mathbb{P}-a.s.$ the weak formulation of $(2.2)$.\\
From the previous theorem we now that any process $\nu_t$ that solves the same identity is such that $\mu_t=\nu_t,\ \mathbb{P}-a.s.$.\\
So any process realizing any limiting law are $\mathbb{P}-a.s.$ equal and so they have the same law. 
\end{proof}
\section{Conjecture on PDE-PDE Limit: Local Interaction}\label{tcseq}
Before analyzing numerically our particle system and retrieve useful insight on turbulence and coagulation formation; In this section we show formally, how we can relate the limiting equation of our particle system to the Smoluchoski Coagulation random PDE (1.1), objective of this paper.\\
We also conjecture the convergence of a sequence of solution $\rho_t^\delta$ of the random Smulochovski PDE, as $\delta\rightarrow0$, in the interaction kernel for the volume.\\
First of all, following the result of (\cite{amann},\cite{1937-1632_2020_12_3319},\cite{article2}), we suppose that our system admits a measure valued solution with  density $\rho_t(x,v)\in C([0,T],L^2(\R^d\times \R_+))\cap L^1(\Omega,C([0,T],L^2(\R^d\times\R_+))$.
\subsection{Approximation of the Local Interaction Kernel}\label{locint}
The rescaled interaction Kernel considered for the volume exchange between particles is of the form
$$
T^{\delta}(x,v,y,w):= \delta^{-3}\mathbb{I}_{B_x(\delta f(v,w))}\ k(v,w),
$$
where we have defined:
\begin{align*}
    &B_x(\delta f(v,w)):=\{y\in\R^d\ :\ |x-y|_d\leq \delta f(v,w)\}, \forall\delta,v,w\in\R;\\
    &f(v,w):=\left(v^{1/3}+w^{1/3}\right)\left(\frac{3}{4\pi}\right)^{1/3};\\
    &k(v,w)\in C^0_b,\ collision\ coefficient.
\end{align*}
Fix $d=3$, we consider now the definition of the mean field equation in weak form, when $\mu_t$ has a density $\rho_t$, for a fixed $\omega,\phi$:
\begin{align}
   \brak{\rho_t,\phi}-\brak{\rho_0,\phi}=&\int_0^t\ \brak{\rho_s,\frac{\varepsilon_x^2}{2}\Delta_x\phi}ds+ \int_0^t\ \brak{\rho_s,\sigma(x,v)\cdot\nabla_x\phi}\xi_s ds \\
   &+\int_0^t\ \brak{\rho_s,\brak{\rho_s,T^\delta J^\phi}_{y,w}}_{x,v}ds\nonumber
\end{align}
where the bracket are to be intend as a $L^2$ product. We see that the only complicated term is the one involving $T^{\delta}$
\begin{align*}
&\left(  \brak{\rho_s,T^\delta J^\phi}_{y,w}\right)  \left(  x,v\right):=\\
&=\int \delta^{-3}\mathbb{I}_{B_x(\delta f(v,w))}\ k(v,w)\left(  \int \underbrace{\left[\mathbb{I}_{(v>=w)}\phi(x,v+w)+\mathbb{I}_{v<w}\phi(y,v+w)-\phi(x,v)-\phi(y,w)\right]}_{J^\phi_{x,v,y,w}}\rho\left(  y,w\right)  dw\right)  dy\\
&=\int_{\R_+} \left(\delta^{-3}\int_{B_x(\delta f(v,w))}J^\phi_{x,v,y,w}\rho_t(y,w)dy\right)k(v,w)dw
\end{align*}
Heuristically, if $\rho$ is regular enough in $w$, since $\phi$ is a test function and therefor regular, we can use a Lebesgue differentiation theorem, and if we take the mean of the integral over the ball as $\delta\rightarrow0$, we have:
$$
a.e.\ x\in\R^d\, \left(\left(\delta f(v,w)\right)^{-3}\int_{B_x(\delta f(v,w))}J^\phi_{x,v,y,w}\rho_t(y,w)dy\right)\rightarrow J^\phi_{x,v,w}\rho_t(x,w),\ as\ \delta\rightarrow0,\ \forall\ w\in\R.
$$
where
$$
J^\phi(x,v,w)=\left[\phi(x,v+w)-\phi(x,v)-\phi(x,w)\right],
$$
and so the convergence of ``$T^{\delta}\ast\rho$" to the new kernel:
\begin{align*}
\left(  \brak{\rho_s,T^\delta J^\phi}_{y,w} \right)  \left(  x,v\right)    &
:=\int_\R \left(f(v,w)\right)^3\left(\left(\delta f(v,w)\right)^{-3}\int_{B_x(\delta f(v,w))}J^\phi_{x,v,y,w}\rho_t(y,w)dy\right)k(v,w)dw\\
&\rightarrow \int_\R\ f(v,w)^3\  k(v,w) J^\phi(x,v,w)\rho_t(x,w) dw,\ a.e. x\in\R^d,\ \forall\ v\in\R_+.
\end{align*}
Resulting in the Random Smoluchoski Coagulation PDE with a turbulence field as in definition $2.2$:
\begin{align}
   \brak{\rho_t,\phi}-\brak{\rho_0,\phi}=&\int_0^t\ \brak{\rho_s,\frac{\varepsilon_x^2}{2}\Delta_x\phi}ds+ \int_0^t\ \brak{\rho_s,\sigma(x,v)\cdot\nabla_x\phi}\xi_s ds \\
   &+\int_0^t\ \brak{\rho_s,\brak{\rho_s,k(v,w)(f(v,w))^3 J^\phi(x,v,w)}_{x,w}}_{x,v}ds\nonumber
\end{align}
\subsection{Conjecture: Converge to Smoluchoski Random Equation}
To conclude, in this section we give a brief idea of a possible convergence of sequence of solution under some additional assumption on the regularity of the random PDE.\\
Consider again the candidate limiting equation in weak form:
\begin{align}
   \brak{\rho_t,\phi}-\brak{\rho_0,\phi}=&\int_0^t\ \brak{\rho_s,\frac{\varepsilon_x^2}{2}\Delta_x\phi}ds+ \int_0^t\ \brak{\rho_s,\sigma(x,v)\cdot\nabla_x\phi}\xi_s ds \\
   &+\int_0^t\ \brak{\rho_s,\brak{\rho_s,k(v,w)(f(v,w))^3 J^\phi(x,v,w)}_{x,w}}_{x,v}ds\nonumber
\end{align}
And the approximate sequence of solutions $\rho_t^\delta$ that solves in weak sense:
\begin{align}
   \brak{\rho_t,\phi}-\brak{\rho_0,\phi}=&\int_0^t\ \brak{\rho_s,\frac{\varepsilon_x^2}{2}\Delta_x\phi}ds+ \int_0^t\ \brak{\rho_s,\sigma(x,v)\cdot\nabla_x\phi}\xi_s ds \\
   &+\int_0^t\ \brak{\rho_s,\brak{\rho_s,T^\delta J^\phi}_{y,w}}_{x,v}ds\nonumber
\end{align}

\begin{proposition}[Limit $\delta$]
Let $\rho_t(x,v)\in C([0,T],L^1\cap L^2(\R^d\times \R_+))\cap L^1(\Omega,C([0,T],L^1\cap L^2(\R^d\times\R_+))$ be a weak solution of $(1.1)$. Suppose that $\mu_0$ has density $\rho_0$ that satisfy hypothesis $(a,b,c)$ as in section.\ref{psm}. Suppose also that the kernel is symmetric and $K(v,w)f(v,w)^3\leq v^\alpha$, with $\alpha\leq 1$.\\
Let $\{\rho_t^\delta\}_\delta$ solutions of the equation with kernel $T^\delta$, then we have
\begin{align*}
    \sup_{t\in[0,T]}\|\rho_t-\rho_t^\delta\|_{L^2(\Omega\times(0,R))}\rightarrow 0,\ as\ \delta\rightarrow0,\ \forall\ R>0,\ \forall \Omega \subset\subset \R,\ \mathbb{P}-a.s.
\end{align*}
\end{proposition}
\begin{proof}
Let's consider a test function $\phi$ and the difference between $\rho_t^\delta,\ \rho_t$ in the weak formulation:
\begin{align*}
\brak{\rho_t-\rho_t^\delta,\phi}&= \int_0^t\brak{\rho_t-\rho_s^\delta,\mathcal{U}_s(x,v)\cdot \nabla_x\left(\phi\right)}ds+\int_0^t\brak{\rho_t-\rho_s^\delta,\Delta_x\left(\phi\right)}ds\\
&\ \ \ +F_{x,v}^{\phi,\delta}(\rho_t^\delta)-F_{x,v}^\delta(\rho_t).
\end{align*}
where the non linear term are
\begin{align*}
&F_{x,v}^\phi(\rho_t)=\int_0^t\brak{\rho_s,\brak{\rho_s,T(v,w)J^{\left(e^{(t-s)A}\phi\right)}}_x}ds\\
&F_{x,v}^{\phi,\delta}(\rho_t^\delta)=\int_0^t\brak{\rho^\delta_s,\brak{\rho^\delta_s,T^\delta J_y^{\left(e^{(t-s)A}\phi\right)}}}ds.
\end{align*}
Consider now the modulus and then, term by term, $\mathbb{P}-a.s.$
\begin{align*}
    &\int_0^t|\brak{\rho_t-\rho_s^\delta,\mathcal{U}_s(x,v)\cdot \nabla_x\left(\phi\right)}|ds\leq \|\xi\|_{L^\infty}\int_0^tC^\phi_{\varepsilon,T}K_\sigma\|\rho_t-\rho_s^\delta\|^2_{2}ds\\
    &\int_0^t|\brak{\rho_t-\rho_s^\delta,\Delta_x\left(\phi\right)}|ds\leq \int_0^t C^\phi_{T}\|\rho_t-\rho_s^\delta\|^2_{2}ds
\end{align*}
while for the other terms we need to add an subtract to get

\begin{align*}
    &F_{x,v}^{\phi,\delta}(\rho_t^\delta-\rho_t)-(F_{x,v}^{\phi,\delta}-F_{x,v}^\phi)(\rho_t):=\\
    &=|\int_0^t\brak{\rho^\delta_s-\rho_s,\brak{\rho^\delta_s,T^\delta(v,w)J^{\left(\phi\right)}}}ds+\int_0^t\brak{\rho_s,\brak{\rho^\delta_s-\rho_s,T^\delta(v,w)J^{\left(\phi\right)}}}ds-(F_{x,v}^{\phi,\delta}-F_{x,v}^\phi)(\rho_t)|\\
    &\leq\left(\sup_t\|\rho_s^\delta\|_1+\sup_t\|\rho_s\|_1\right)\int_0^tC^\phi_{\Omega,R}\|\rho^\delta_s-\rho_s\|_2^2ds+|(F_{x,v}^{\phi,\delta}-F_{x,v}^\phi)(\rho_t)|
\end{align*}
For the last term we have
\begin{align*}
    &(F_{x,v}^{\phi,\delta}-F_{x,v}^\phi)(\rho_t)=\\
    &=\int_0^t|\brak{\rho_s,\brak{\rho_s,T^\delta\ J^{\left(\phi\right)}}-\brak{\rho_s,T J^{\left(\phi\right)}}_x}|ds\leq\\
    &\int_0^t|\brak{\rho_s,\int_0^R k(v,w)\left(\int_\Omega \rho_t T^\delta \tilde{J}^\phi dy-\rho_sT J^{\left(\phi\right)}\right)dw}ds\leq C^\phi\int_0^t|\brak{\rho_s,\int_0^R k(v,w)Ld^\delta(\rho_t T J^\phi)dw}ds
\end{align*}
Where
$$
Ld^\delta(\rho_tTJ^\phi):=\left(\int_\Omega \rho_t T^\delta \tilde{J}^\phi dy-\rho_sT J^{\left(\phi\right)}\right)\rightarrow0,\ in L^\infty_vL^1_x,\ as \delta\rightarrow0
$$
thanks to Lebesgue differentiation theorem and bound on the volume component of the density. Putting all together we have:
\begin{align*}
    & C^\phi\int_0^t|\brak{\rho_s,\int_0^R k(v,w)Ld^\delta(\rho_t T J^\phi)dw}ds\\
    &\leq \int_0^tc^\phi M_1(v\rho_t)\sup_{w\in(0,R)}\|Ld^\delta(\rho_tT)\|\leq c^\phi M_1(v\rho_0)\sup_{w\in(0,R)}\|Ld^\delta(\rho_tT)\|ds:=\int_0^tA_\delta C^\phi_{R,\rho_0,\Omega}ds
\end{align*}
where the last inequality follows from the bound on the microscopic kernel $k$. If we put all the estimate together, taking the sup over all test functions with $L^2$ norm equal 1, we obtain with Gronwall the following
$$
\sup_t\|\rho_t^\delta-\rho_t\|_2^2\leq \left(\sup_t (A_\delta) T\right)\exp(C^\phi_{R,\Omega,\rho_0})\rightarrow0,\ as\ \delta\rightarrow0.
$$
With this we conclude the proof.
\end{proof}
\section{Numerical Simulation: First Time Rain Generation}\label{numan}
We conclude this work with a numerical simulation, with a subsequent analysis, of the particle system in two different setting to better understand the role of turbulence in the generation of rain inside a cloud formation.\\
In detail, what we are interested in the study of the first formation time 
$$\tau_f:=\inf\{t\geq 0 | \exists i\in N_t,\  R^i_t>=R_{rd}\},$$
where $R_{rd}:=0.0004$ is the radius of the typical raindrop.\\
What we obtain is a decay dependence of the aforementioned exit time in different setting: a pure Brownian case, and a random field driven particles.\\
\\
In the general setting we consider the following positional equation for the particle system
\begin{align}
    &dX^i_t=\varepsilon_{rf}\sum_{k\in K}\sigma_k(X_t^i)\xi_t^k dt-\hat{e}_2f(V_t^i)dt+\varepsilon_{bm} \sigma dB_t^i,
\end{align}
for $i=1,...,N$, which we have simulated with an Eulero-Maruyama scheme.\\
$\varepsilon_{rf},\varepsilon_{bm}\in\{0,1\}$ select the setting in which we investigate the formation time $\tau_f$.\\
The domain in which we have done our simulation is $[-2,2]^2$ (that is intended to be measured in $cm$) with periodic condition on the border, so that it represent a zoom inside the cloud and we assume that as the particles move they are replaced with probability one with particles of almost the same size.\\
Concerning the velocity random field in which the particles moves, it is defined as the stochastic process $\mathcal{U}\left(  t,x\right)  $%
\[
\mathcal{U}\left(  t,x\right)  =\sum_{k\in K}\sigma_{k}\left(  x\right)  \xi_{t}^{k}%
\]%
\[
\sigma_{k}\left(  x\right)  =\frac{1}{2\pi}\frac{\left(  x-x_{k}\right)
^{\perp}}{\left\vert x-x_{k}\right\vert ^{2}},\ \left(  a,b\right)  ^{\perp}=\left(  b,-a\right)%
\]%
where points $x_{k}$ are fixed and selected with a uniform distribution over the considered domain, representing vortex generated by the turbulent flow in which the droplet are submerged. We have used the notation such that $\sigma_k(x_k)=0$.\\
The real-valued stochastic processes $\xi_t^k$ are Ornstein-Uhlenbech that satisfy:
\[
d\xi_{t}^{k}  =-\lambda\xi_{t}^{k}dt+\lambda dB_{t}^{k},\ \ \  
\xi_{0}^{k}  =0,\ k  =1,...,K.\\
\]%
which we have also simulated with an Eulero-Murayama scheme and used the result to compute $(4.1)$.\\
The initial condition of the particle, $X_0^i$, are sampled from a uniform distribution over the domain, while the initial volumes $V_0^i>0$ are selected uniformly in the range of a large raindroplet ($\sim 10^3m)$.\\
We specify also that we assume the raindrops are spherical and as such the volume can be represented as
$$
\forall i=1,...,N,\ V_t^i\sim\frac{4}{3}\pi(R_t^i)^3.
$$
While the function $f$ representing the limit velocity induced by the mass of the raindrop was simulated with a standard logistic function such that the limit is exactly the terminal velocity as in $[\cdot]$.
\subsubsection{Stochastic coalescence:} 
Two real-droplets may collide and coalesce completely into one big real-droplet and this process is responsible for precipitation development.\\ 
Our primary postulate, which is commonly used in many previous studies, is that the coalescence process can be described in a probabilistic way.\\
So, what we expect is that our $V_t^{i,N}$, of the $i-th$ particle, changes not in a continuous way, but with jumps proportional to the particle that $i$ interact with.\\
With this we mean that, when two particle interact with each other there is a chance that the volume of the $i-th$ particle absorb or get absorbed by the other particle. So that the volume of a particle can increase with a certain probability, only when collide with other rain droplets.\\
\\
For the sake of 
completeness we need to define a $P_{mean}$ 
that represent a threshold that encapsulate the possibility that two rain droplets can interact or bounce from each other, and if extracting a random number we surpass this threshold that coalescence can happen, otherwise the particles get along with their motion.\\
To state it clearly this is how our coalescence works:\\
\paragraph{\textbf{Algorithm for stochastic coalescence}} 
\begin{enumerate}
    \item Let $x_i,\ i=1,...,N$ the particles of the system, and denote $R_i$ their radius;
    \item For every couple $(i,j)$  such that $|x_i-x_j|\leq R_i+R_j$, we generate a number $\phi_i^j\in (0,1)$ with a uniform distribution;
    \item If $\phi_i^j>P_{mean}$, then the coalescence happens in the following way:
    \begin{itemize}
        \item If $V_t^{i,N}>=V_t^{j,N}$ then
        \begin{itemize}
            \item $V_{t}^{i,N}= V_t^{i,N} + V_t^{j,N}$;
            \item $V_{t}^{j,N}=0$ and the particle $j$ is removed from the system.
        \end{itemize}
        \item Vice versa if  $V_t^{j,N}>V_t^{i,N}$;
    \end{itemize}
    \item If $\phi_i^j<P_{mean}$, then the raindroplet bounced of from each other.
\end{enumerate}
\subsection{Pure Brownian Movement}
In this small section we analyze the particle system under the action of indendent brownian motion attachted to the droplets.\\
We set $\varepsilon_{rf}=0,\varepsilon_{bm}=1$ in $(4.1)$ and we reduce the model to the motion equation given by:
\begin{align}
    &dX^i_t=-\hat{e}_2f(V_t^i)dt+\sigma dB_t^i,
\end{align}
where $X_0^i$ are drawn from a uniform distribution over the domain $[-2,2]^2$.\\
The radius of the particles are selected in such a way that a raindrop is created as soon as the value it the threshold 
$$
R^i_t\sim 0.004 m
$$
and the radius $R_0^i$ of the starting rain droplets  are drawn from a uniform distribution between values $10^{-1},\ 10^{-2}$ smaller than a real raindrop.\\
Finding the right time step was challenging, since the position of the particle is discrete and as such we could lose interaction between particle in the motion of the rain droplets, causing numerical error and inconsistency in the behaviour of the formation time.\\
While the code can run on a large number of particle, to analyze efficiently the result we set our parameters as follows:
$$
dt=0.0001,\ \ N=1000,\ \ T_{MaxIt}=3000,
$$
where $N$ represents the initial number of particle and $T_{MaxIt}$ the maximal iteration of the system.\\
To simplify the analysis we postpone the investigation of the coefficient $P_{mean}$ assuming it is independent from the fact that the motion is subjected under the action of a random field or independent Brownian Motion.\\
In such a way we can assume without loss of generality that $P_{mean}=0$, and the coalescence algorithm is as follows:
\begin{enumerate}
    \item Let $x_i,\ i=1,...,N$ the particles of the system, and denote $R_i$ their radius;
    \item For every couple $(i,j)$  such that $|x_i-x_j|\leq R_i+R_j$:
    \begin{itemize}
        \item If $V_t^{i,N}>=V_t^{j,N}$ then
        \begin{itemize}
            \item $V_{t}^{i,N}= V_t^{i,N} + V_t^{j,N}$;
            \item $V_{t}^{j,N}=0$ and the particle $j$ is removed from the system.
        \end{itemize}
        \item Vice versa if  $V_t^{j,N}>V_t^{i,N}$;
    \end{itemize}
\end{enumerate}
We wan to analyze the formation time $\tau_{f}$ dependent on the intensity $\sigma$ of the independent Brownian motions.\\
To do so, for each parameter, we have simulated ten times the particle system, collecting the time formation in a time series from which we extract the mean time. In detail we have, calling $N_r$ the number of repetition, $T^i_f$ the formation time of the $i-th$ simulation, the following:
$$
N_r=10,\ \ T_{mean}=\mathbb{E}[T_f]\sim\frac{1}{N_r}\sum_{i=1}^{N_{r}} T_f^i, \forall \sigma \in \Sigma.
$$
Where $\Sigma:=\{0.1,0.2,0.3,0.4,0.5,0.6,0.7,0.8,0.9,1\}$, is a partition of the reasonable interval, $[0,1]$, for the intensity of a realistic Brownian motion.
\subsubsection{A simplified theoretical reasoning}
Before showing the numerical result, we show a theoretical reasoning on a simplified model that can explain the expected result in such a model.\\
Let us consider two particles moving subjected to independent Brownian motion with different initial condition: $\sigma W^1_t+x_1$, $\sigma W_t^2+x_2$ on the unitary Torus.\\
Call $x_0 = x_1-x_2$ and define the random time
\[
\tau_{\sigma,x_0}:=\inf\{t\geq0: |\sigma W_t^1-\sigma W_t^2 +x_0|\leq \varepsilon\}
\]
which is linked to our formation time, since coalescence happens when two raindrops are near depending on a setted threshold.\\
Then we ask the question: how $\sigma$ impact the collision time of particles?\\
This quantity cannot be independent from $\sigma$, indeed call $B_t = W_t^1- W^2_t$, this is a new Brownian motion, with intensity dependent from the dimension (but fixed), so without loss of generality we fixed it at 1.\\
Thus we consider the quantity $\tau_{\sigma,x_0}$:
\begin{align*}
\tau_{\sigma,x_0}:&=\inf\{t\geq0: |\sigma B_t +x_0|\leq \varepsilon\}\\
&=\inf\{t\geq0: |\sigma B_{\sigma^2 t/\sigma^2} +x_0|\leq \varepsilon\}\\
&=\inf\{t\geq0: |W_{\sigma^2t} +x_0|\leq \varepsilon\}\\
&=\tau_{1,x_0}/\sigma^2.
\end{align*}
Here we have used the auto-similarity of the Brownian motion.\\
What we recover is that there is an inverse quadratic dependence of the formation time with the intensity of the Brownian motion moving the particles. That we can express in a sintetic form as
$$
\mathbb{E}\left[\tau_{\sigma,x_0}\right]\sim \sigma^{-2}.
$$
\subsubsection{Results on Regression}
Let's start showing the result table of the mean formation time and the standard deviation, in dependence of Brownian intensity.\\ (Note that to obtain a time measure we need to multiply the epoch with the step time $dt$).\\
Following the previous reasoning we expect that our formation time depend with a square inverse power with the intensity of the diffusion
$$
\mathbb{E}\left[\tau_{f}(\sigma)\right]\sim \sigma^{-2},
$$
where we have highlighted the dependence on $\sigma$.\\
\begin{center}
\begin{tabular}{ |c|c|c| } 
 \hline
 Brownian intensity &  Mean time & Standard Deviation \\ 
 1 & 636 & 114.9043 \\ 
 0.9 & 641.4 & 119.8025\\
 0.8 & 691.5 & 93.29014 \\ 
 0.7 & 734.9 & 170.0514\\
 0.6 & 808 & 165.7607 \\ 
 0.5 & 917.6 & 234.5115 \\
 0.4 & 1086.9 & 312.3816 \\ 
 0.3 &  1569.8 & 260.9294\\
 0.2 &  2896.5 &  387.9338;\\ 
 0.1 &  4567 & 434.2\\
 \hline
\end{tabular}
\end{center}
\begin{figure}[!tbp]
  \centering
  \begin{minipage}[b]{0.7\textwidth}
\includegraphics[width=0.7\textwidth]{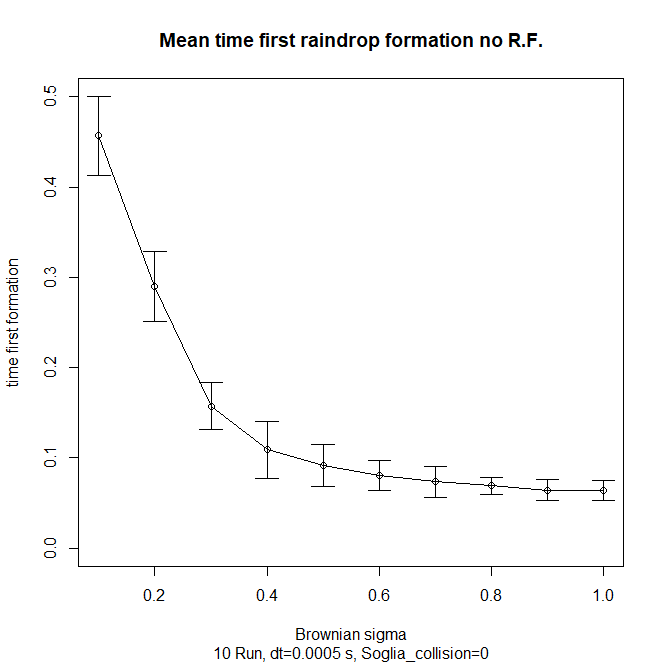}
  \end{minipage}
\hfill
  \begin{minipage}[b]{0.7\textwidth}
    \includegraphics[width=0.7\textwidth]{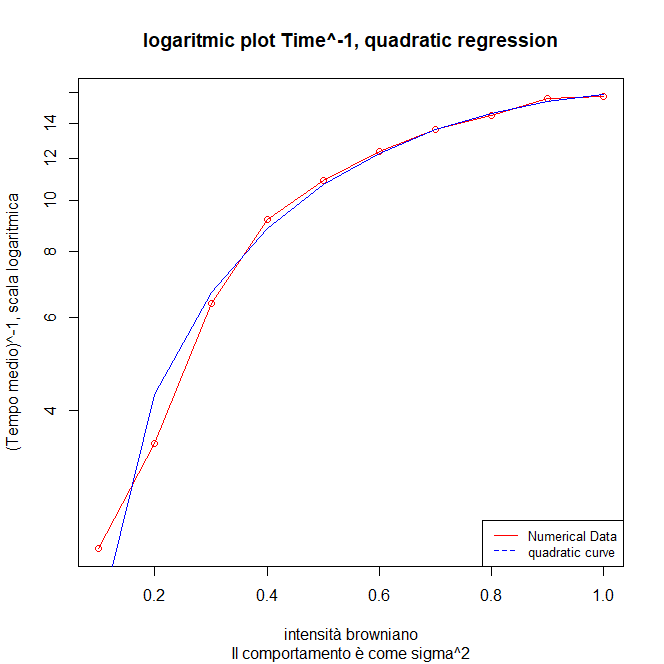}
  \end{minipage}
  \hfill
  \begin{minipage}[b]{0.7\textwidth}
\includegraphics[width=0.7\textwidth]{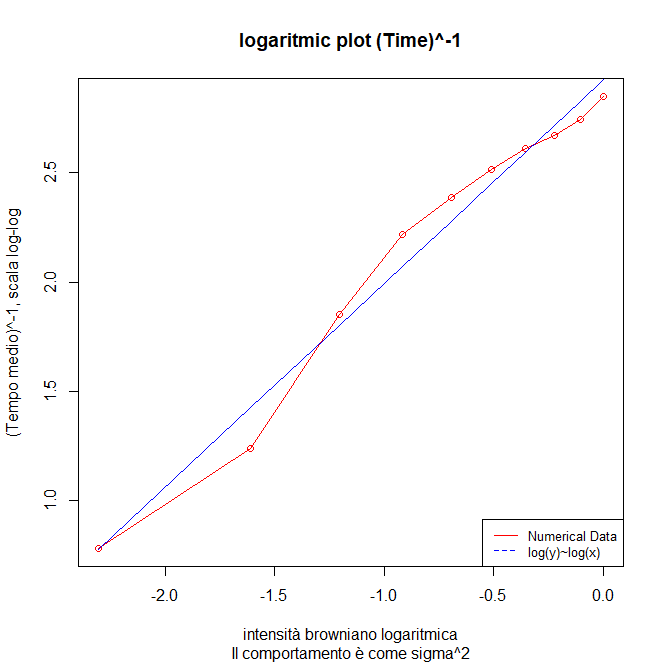}
  \end{minipage}
  \caption{From top to bottom: (a) Plot of Brownian Intensity/Formation time that shows the time decay in dependence with increase in $\sigma$, intensity; (b) Linear-log plot: quadratic regression, in red the interpolated numerical data and in blue the regression curve;(c) log-log plot: logarithmic regression, in red the interpolated numerical data and in blue the regression curve.}
  \label{f1}
\end{figure}
As we can see from the table, there is a significance decay in the mean time as the intensity grows, and, as expected, the standard deviation became smaller.\\
As $\sigma$ grows the initial condition of the system became less and less important for the mean formation time, as we can see as the standard deviation decrease as $sigma$ increases.\\
To understand the decay we proceed our analysis with a quadratic and logarithmic regression on the model.\\ 
In $Figure\ 1$ we can see plotted, $(a)$, the decay of mean formation time that shows the time decay in dependence with the increase in $\sigma$, intensity of the Brownian motions. Again we can appreciate the stabilization of the randomness in the standard deviation, plotted as black bar in figure.\\
\\
First we performed a quadratic regression, for which we show the results in fig.\ref{f1}(b) as a Lin-Log plot. We can already notice how the behavior of the numerical data, interpolated in red, fit the regression line in blue and the most significant error are in the very small regime of $\sigma\sim 0.1$, since the particle are almost still in their motion and the dependence from the randomly generated data is stronger.\\
The regression show a residual plot without any particular structure, with a residual standard error equal to 0.5232 and the R-square equal to 0.992 and a summary that agrees with our conjecture on the quadratic decreasing of time respect to the intensity of the noise.\\
\\
\textbf{Quadratic Regression Summary:}
\begin{verbatim}
Residual standard error: 0.5232 on 7 degrees of freedom
Multiple R-squared:  0.992,	Adjusted R-squared:  0.9897 
F-statistic: 431.5 on 2 and 7 DF,  p-value: 4.673e-08.
\end{verbatim}

To validate our results we concluded the analysis with a logarithmic regression to capture lower order power and fluctuation in the intensity dependence of the exit time. The results in fig.\ref{f1}(c) as a Log-Log plot confirm how the behavior of the numerical data, interpolated in red, fit the regression line, in blue, in accordance to the quadratic regression. Small fluctuations around the regression line can be appreciated, showing lower order dependencies on the intensity, but the residual plot showing no structures with a standard error of 0.1 and a R-squared equal to 0.979, agrees with the hypothesis $\mathbb{E}[\tau_f^\sigma]\sim\sigma^{-2}$.\\
\\
\textbf{Logarithmic Regression Summary:}
\begin{verbatim}
Residual standard error: 0.1061 on 8 degrees of freedom
Multiple R-squared:  0.979,	Adjusted R-squared:  0.9764 
F-statistic: 373.6 on 1 and 8 DF,  p-value: 5.329e-08
\end{verbatim}

\subsection{Random Field Movement}
In this section we are going to analyze the particle system under the action of the Velocity Random Field.\\
Here we present a brief numerical study for the particle system under theoretical study, to motivate future work on the stochastic coalescence and turbulence formation in this interacting system.\\
We considered the following system for the position of the particles
\begin{align}
    &dX^i_t=\sum_{k\in K}\sigma_k(X_t^i)\xi_t^k dt-\hat{e}_2f(V_t^i)dt+\varepsilon dB_t^i
\end{align}
for $i=1,...,N$, which we have simulated with an Eulero-Maruyama scheme.\\
The domain in which we have done our simulation is $[-2,2]^2$ (that is intended to be measured in $cm$) with periodic condition on the border, so that it represent a zoom inside the cloud and we assume that since the particle moves they are replaced with probability one with particles of almost the same size.\\
Concerning the velocity random field in which the particles moves, it is defined as the stochastic process $\mathcal{U}\left(  t,x\right)  $%
\[
\mathcal{U}\left(  t,x\right)  =\sum_{k\in K}\sigma_{k}\left(  x\right)  \xi_{t}^{k}%
\]%
\[
\sigma_{k}\left(  x\right)  =\frac{1}{2\pi}\frac{\left(  x-x_{k}\right)
^{\perp}}{\left\vert x-x_{k}\right\vert ^{2}},\ \left(  a,b\right)  ^{\perp}=\left(  b,-a\right)%
\]%
where points $x_{k}$ are fixed and selected with a uniform distribution over the considered domain. We have used the notation such that $\sigma_k(x_k)=0$.\\
The real-valued stochastic processes $\xi_t^k$ are Ornstein-Uhlenbech that satisfy:
\[
d\xi_{t}^{k}  =-\lambda\xi_{t}^{k}dt+\lambda dB_{t}^{k},\ \ \  
\xi_{0}^{k}  =0,\ k  =1,...,K.\\
\]%
which we have also simulated with an Eulero-Murayama scheme and used the result to compute $(5.1)$.\\
The initial condition of the particle, $X_0^i$, are sampled from a uniform distribution over the domain, while the initial volumes $V_0^i>0$ are selected uniformly in the range of a large raindroplet ($\sim 10^3m)$.\\
We specify also that we assume the raindrop are spherical and as such the volume can be represented as
$$
\forall i=1,...,N,\ V_t^i\sim\frac{4}{3}\pi(R_t^i)^3.
$$
While the function $f$ representing the limit velocity induced by the mass of the raindrop was simulated with a standard logistic function such that the limit is exactly the terminal velocity as in $[\cdot]$.
\subsubsection{Stochastic coalescence:} 
Two real-droplets may collide and coalesce completely into one big real-droplet and this process is responsible for precipitation development.\\ 
Our primary postulate, which is commonly used in many previous studies, is that the coalescence process can be described in a probabilistic way.\\
So, what we expect is that our $V_t^{i,N}$, of the $i-th$ particle, changes not in a continuous way, but with jumps proportional to the particle that $i$ interact with.\\
With this we mean that, when two particle interact with each other there is a chance that the volume of the $i-th$ particle absorb or get absorbed by the other particle. So that the volume of a particle can increase with a certain probability, only when collide with other rain droplets.\\
\\
For the sake of 
completeness we need to define a $P_{mean}$
that represent a threshold that encapsulate the possibility that two rain droplets can interact or bounce from each other, and if extracting a random number we surpass this threshold that coalescence can happen, otherwise the particles get along with their motion.\\
To state it clearly this is how our coalescence works:\\
\paragraph{\textbf{Algorithm for stochastic coalescence}} 
\begin{enumerate}
    \item Let $x_i,\ i=1,...,N$ the particles of the system, and denote $R_i$ their radius;
    \item For every couple $(i,j)$  such that $|x_i-x_j|\leq R_i+R_j$, we generate a number $\phi_i^j\in (0,1)$ with a uniform distribution;
    \item If $\phi_i^j>P_{mean}$, then the coalescence happens in the following way:
    \begin{itemize}
        \item If $V_t^{i,N}>=V_t^{j,N}$ then
        \begin{itemize}
            \item $V_{t}^{i,N}= V_t^{i,N} + V_t^{j,N}$;
            \item $V_{t}^{j,N}=0$ and the particle $j$ is removed from the system.
        \end{itemize}
        \item Vice versa if  $V_t^{j,N}>V_t^{i,N}$;
    \end{itemize}
    \item If $\phi_i^j<P_{mean}$, then the raindroplet bounced of from each other.
\end{enumerate}
Since our objective is study the mean time of such rain formation, we fixed\\ $P_{mean}=0$ as we already expect some fluctuation in the results given by the initial condition, and postpone to later studies the understanding of such physical quantity.
\subsubsection{Random Field:}\textit{Theoretical Conjecture}.
Before moving to the numerical results, we briefly illustrate theoretical conjecture on what we expect to find in the next section.\\
Following the idea from (\cite{f1},\cite{flandoli2021mixing},\cite{f2}) in which a fluid dynamical equations with transport noise converges to deterministic viscous equations such that the viscosity is enhanced by the turbulent random field.\\
We expect that, albeit less impactful than independent Brownian motion attach to each particles, the first exit time depends on the random field parameters as an inverse power $\alpha\in(-2,0)$.\\
\\
Consider our Smoluchoski equation with random field
$$
\partial_t \rho_t+ \sum_{k\in K}\left(\sigma_k(x,v)\cdot \nabla_x\rho_t\right)\xi_t=F_v(\rho_t)
$$
We expect that, under suitable assumption, the solution of such coagulation PDE is close to the solution of the diffusion Smoluchoski equation
$$
\partial_t \overline{\rho}_t+ \frac{\kappa^2}{2}\Delta\overline{\rho}_t=F_v(\overline{\rho}_t),
$$
where $\kappa:=\kappa(K,\lambda)$ depends on the parameters of the system in study.\\
From the theory on the Brownian motion case we expect that
$$
\mathbb{E}\left[\tau_f(\kappa)\right]\sim p(\kappa)^{-1}
$$
where $p(\kappa)$ is at most a polynomial of degree $2$.
\begin{figure}[h]
\centering
\includegraphics[width=\textwidth]{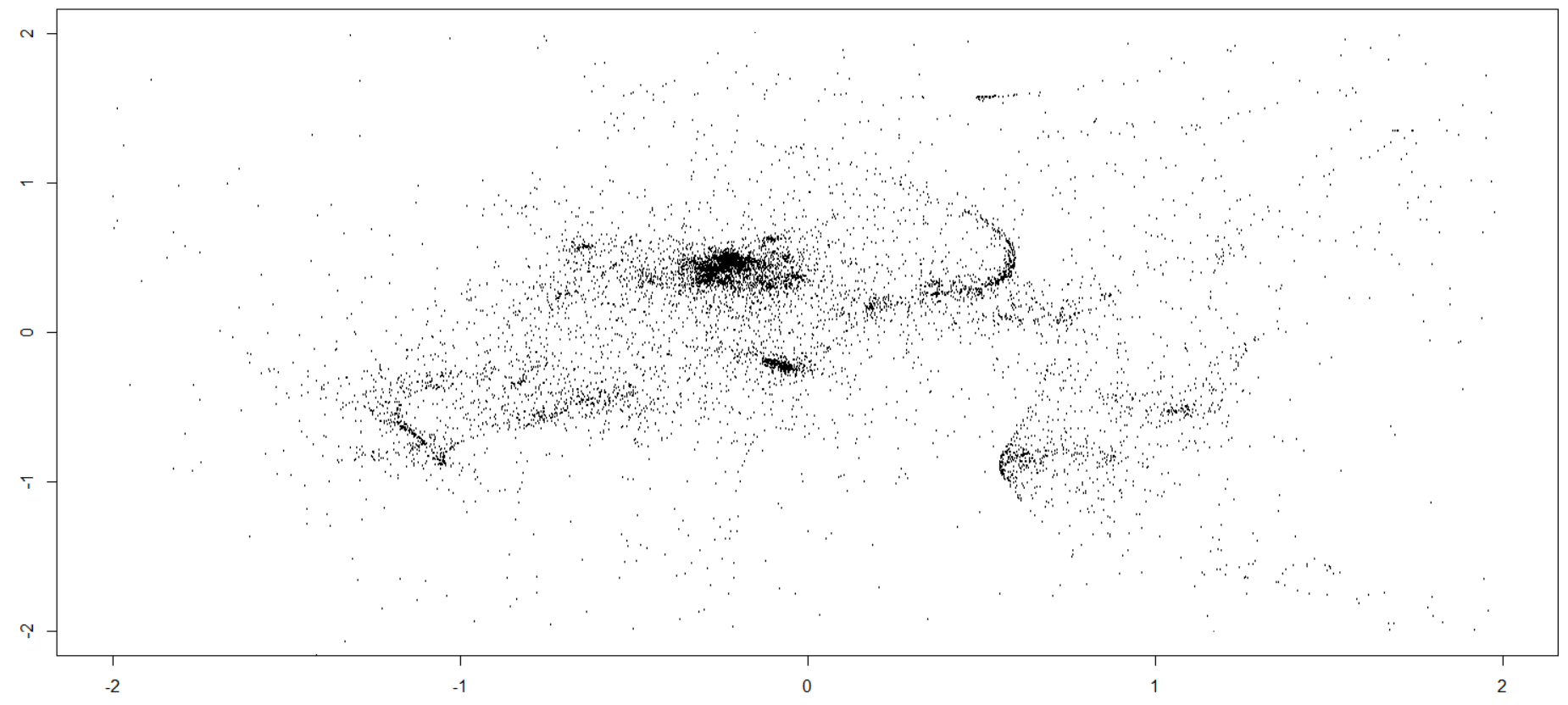}
\caption{Snapshot of particle system dynamics: we can observe the cluster formation obtained from the presence of the vortex. We expect that such cluster, and the randomness of the Ornstein–Uhlenbeck process, are the reason for the decrease in time of rain formation.}
\end{figure}

\subsection{Random Field Analysis: Non-Linear Regression}
We divide the analysis in two main parts. We start fixing the intensity of the Ornstein-Ulhembeck processes and varying the number of vortex in the cloud; later we fix the vortex number and move the intensity of the environmental stochastic processes.\\
Following the previous reasoning we expect that our formation time depend with at most a square inverse power with the number of vortex and intensity of the random processes 
$$
\mathbb{E}\left[\tau_{f}(|H|,\lambda)\right]\sim p(|H|,\lambda)^{-1},
$$
where we have highlighted the dependence on the parameter.\\
From the theoretical conjecture in the previous subsection, we expect that as vortex increase, so does linearly the covariance matrix of the noise, resulting in a square decay for the time.\\ While $\lambda$ model the  temporal structure of the phenomena, thus we expect that the effect is smaller on the vector fields $\sigma$ with a inverse decay trend $\sim \sqrt{\lambda}$.
\subsubsection{Fixed $\lambda$, Vortex Moving}
We analyze here the case in which the intensity of the noise is fixed and the number of vortex in the field changes. We'll see, as claimed at the start of the section that the decay is visible when we increase the number of field $\sigma_j$ interacting with the particles. We show here regression that support the conjecture of a polynomial decay of order $2$.\\ We note that the initial condition produce a very fluctuating system when the vortex number is small and the particles moves very slowly. This is expected, since the collision are reduced and the random initial condition characterize the system a lot. This is show in the high standard deviation that decreases as soon as the vortex number increase.\\
We fix $\lambda$ at 1500, to have enough speed in the particle and make sure that the collisions of droplet happen fast enough to be appreciated, in particular the choice is made to obtain a fast enough system for the required simulation to be made even with very few vortex-like structure.\\
We first performed a quadratic regression from which the residual plot showed no sign of residual structure. The analysis show a great degree of understanding of the data, as such R-squared is equal to $\sim0.98$ and the p-value suggest that our hypothesis on the decay is validated. We report here the complete summary of the regression performed in Fig.4 (left).\\
\\
\textbf{Quadratic Regression Summary Moving Vortex Number:}
\begin{verbatim}
Residual standard error: 0.4133 on 8 degrees of freedom
Multiple R-squared:  0.9798,	Adjusted R-squared:  0.9748 
F-statistic: 194.2 on 2 and 8 DF,  p-value: 1.658e-07
\end{verbatim}
\begin{figure}[!tbp]
  \centering
  \begin{minipage}[c]{0.45\textwidth}
    \includegraphics[width=\textwidth]{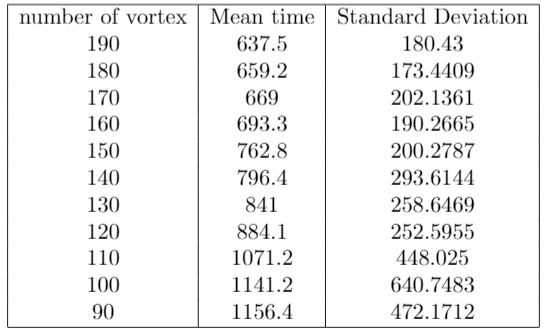}
  \end{minipage}
  \hfill
  \begin{minipage}[c]{0.45\textwidth}
    \includegraphics[width=\textwidth]{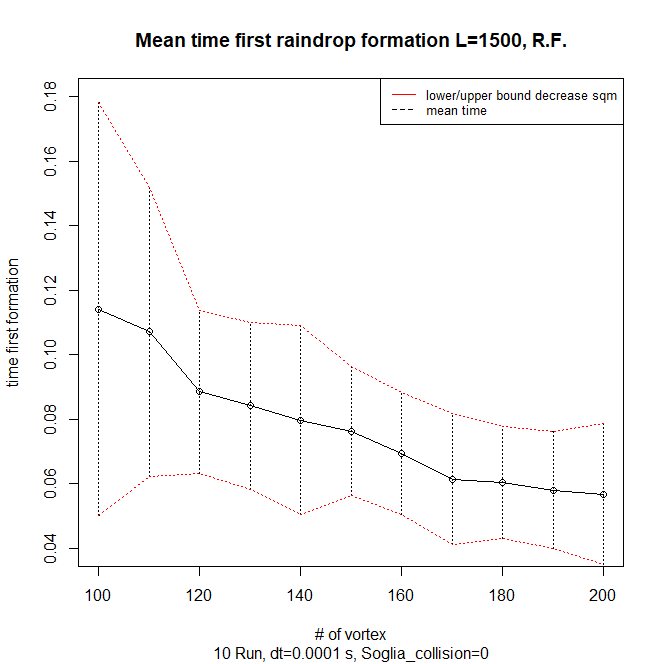}
    
  \end{minipage}
  \caption{Left: table of time decay as vortex number increase and as such the number of vortex in the system grows. Mean time and Standard deviation decays accordingly to the growth of $\#Vortex$; Right: The black line represent the mean time varying the number of vortex, while the red strip are the upper bound and lower bound of the standard deviation.}
\end{figure}

To conclude the analysis we perform a logarithmic regression, that we plotted in Fig.3 (right) in a log-log scale. The analysis is consistent with the quadratic regression performed and showed a good R-squared value. No significant structures were found in the residual plot and the value of the regression supported our hypothesis of a quadratic decay in the time, with small fluctuation due to the randomness of the initial condition that stabilize at higher vortex number.\\
\\
\textbf{Logarithmic Regression Summary Moving Vortex Number:}
\begin{verbatim}
Residual standard error: 0.4133 on 8 degrees of freedom
Multiple R-squared:  0.9798,	Adjusted R-squared:  0.9748 
F-statistic: 194.2 on 2 and 8 DF,  p-value: 1.658e-07
\end{verbatim}

\begin{figure}[!tbp]
  \centering
  \begin{minipage}[b]{0.45\textwidth}
    \includegraphics[width=\textwidth]{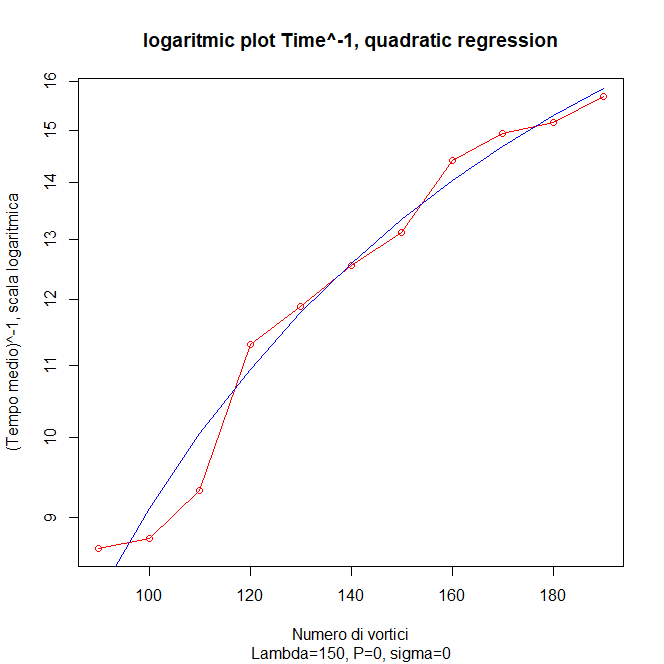}
  \end{minipage}
  \hfill
  \begin{minipage}[b]{0.45\textwidth}
    \includegraphics[width=\textwidth]{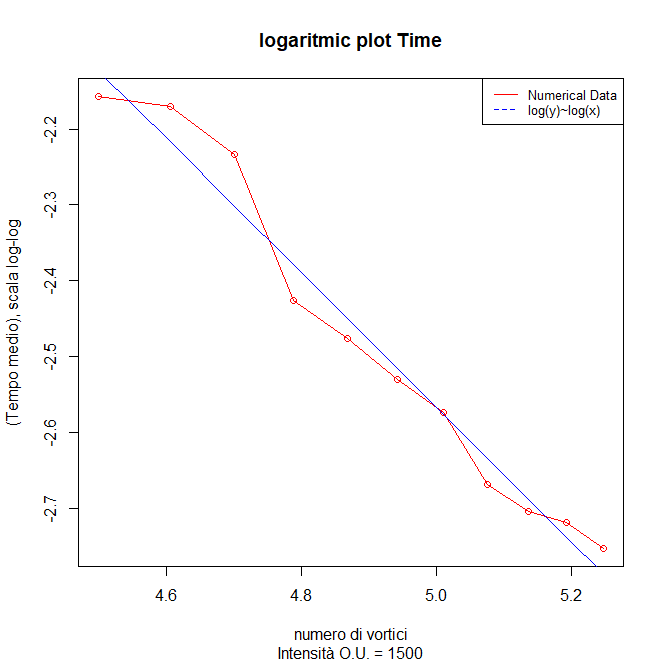}
  \end{minipage}
  \caption{Left:Lin-Log plot of quadratic regression with fixed intensity and varying vortex number; Right: Log-Log plot for Logarithmic regression with fixed intensity and varying vortex number.}
\end{figure}

\subsubsection{Fixed Vortex, $\lambda$ Moving}
We analyze here the case in which vortex are fixed and the intensity of the Random field changes. We'll see, as claimed at the start of the section that the quadratic decay is no longer visible and the initial condition produce a very unstable numerical system.\\
We fix the vortex number at 200, the choice is made to obtain a fast enough system for the required simulation to be made even with very slow random field intensity, but selecting also a number of vortex that already showed a nice decay in time.\\

\begin{figure}[!tbp]
  \centering
  \begin{minipage}[b]{0.5\textwidth}
    \includegraphics[width=\textwidth]{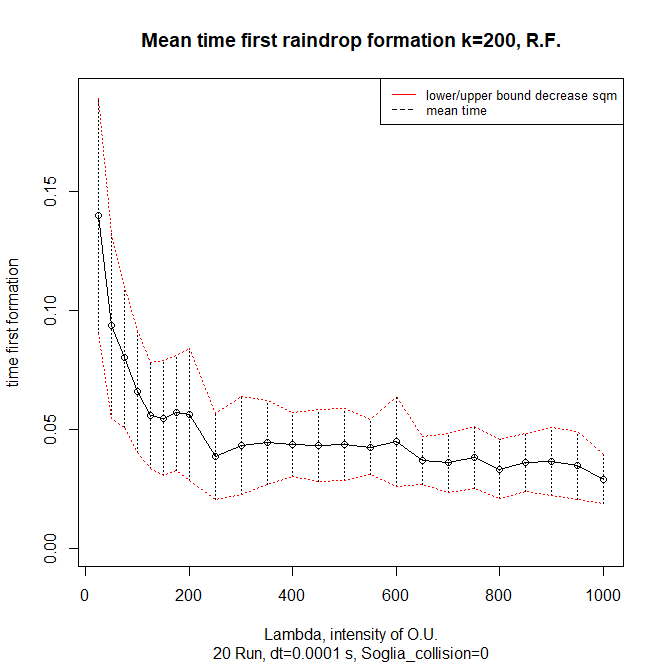}
  \end{minipage}
  \hfill
  \begin{minipage}[b]{0.45\textwidth}
    \includegraphics[width=\textwidth]{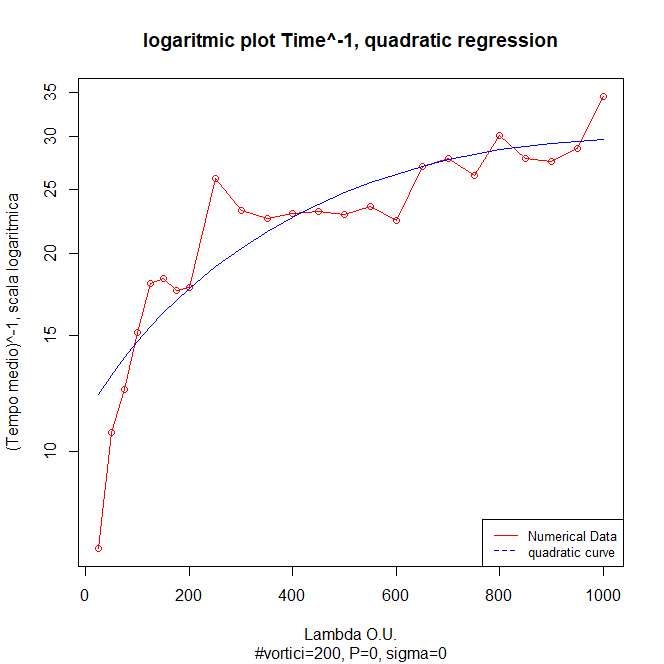}
  \end{minipage}
  \caption{Left: numerical plot of mean time with standard deviation. It is shown a decay in time when the intensity increase, with also a decrease in the standard deviation; Right: Lin-Log plot quadratic regression of mean time with varying $\lambda$, fixed vortex number.}
\end{figure}

We first performed a quadratic regression from which the residual plot showed no sign of residual structure. The residual error and the standard deviation was overall higher, even with high intensity parameter that should make coagulation faster. This suggested, as we expected from the theoretical conjecture, that the influence of the intensity in the time structure is less strong than the number of vortex-like structures in the field.\\
The analysis show a less degree of understanding of the data, and while R-squared is still high, we expect that power of lower order describe the decay curves in a more robust way. We report here the complete summary of the regression performed in Fig.5 (right).\\
\textbf{Quadratic Regression Summary Moving $\lambda$}
\begin{verbatim}
Residual standard error: 2.747 on 21 degrees of freedom
Multiple R-squared:  0.8392,	Adjusted R-squared:  0.8239 
F-statistic:  54.8 on 2 and 21 DF,  p-value: 4.633e-09
\end{verbatim}

To see if a regression with lower power of the parameters can explains better the decay of time formation, we performed a logarithmic regression, plotted in Fig.6. From the summary below, we see that more structure is captured when this type of analysis is performed, suggesting that a non linear regression of the form $\sim a/(1+b\cdot\lambda)$ is the right expected decay for the time as a function of the intensity parameter.\\
\textbf{Logarithmic Regression Summary Moving $\lambda$}
\begin{verbatim}
Residual standard error: 0.1084 on 22 degrees of freedom
Multiple R-squared:  0.9145,	Adjusted R-squared:  0.9106 
F-statistic: 235.3 on 1 and 22 DF,  p-value: 3.135e-13
\end{verbatim}

\begin{figure}[h]
\centering
\includegraphics[width=0.5\textwidth]{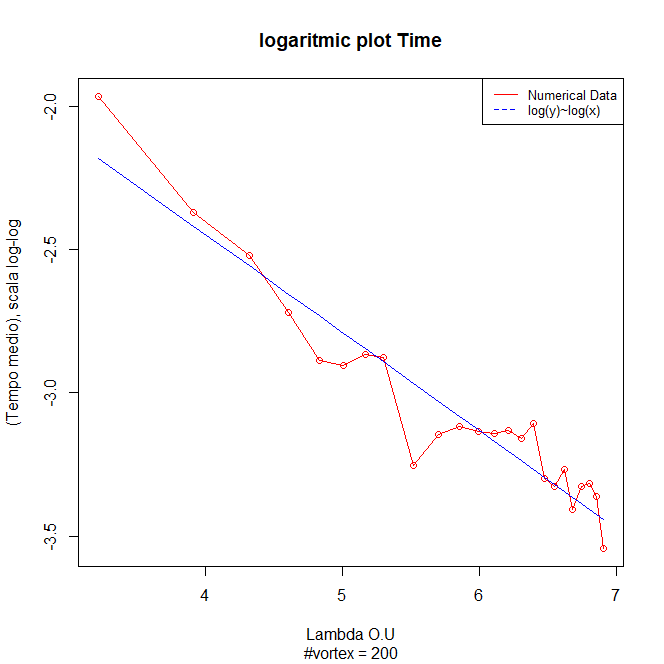}
\caption{Logarithmic Regression: Log-Log plot regression mean time with varying $\lambda$, fixed vortex number.}
\end{figure}

To this end, to get a more precise estimation on the behaviour of the first formation time $\tau_f(\lambda)$, with fixed vortex, we performed a non linear regression on a more refined sample of intensity.\\
As show in Fig.7 we iterate the system for intensity $\lambda\in[100,1100]$ with step of $10$.\\
After the logarithmic regression, we expect that we have a dependence on the parameter of the form $\sim \frac{a}{(b+\lambda)}$. For this reason we set $a_{in}=0.08,\ b_{in}=0.048$ starting parameter and we perform a regression that shown in Fig.7.\\
As we can see from the summary below the residual error equal to 0.0029 with a correlation with the numerical data equal to 0.79. As we see no structure is present in the residual plot and the decay is captured by the regression curve.\\
In conclusion when the intensity of the noise increase we still see a decay in time, but since the noise interact only in the time structure the decay has lower degree.\\
\\
\textbf{Non Linear Regression $\sim a/(1+b\cdot\lambda)$ Summary:}
\begin{verbatim}
residual sum-of-squares: 0.002907
Achieved convergence tolerance: 1.489e-06
cor(NumericalData,NonLinearModel)=0.7997127
\end{verbatim}
\begin{figure}[!tbp]
  \centering
  \begin{minipage}[b]{0.5\textwidth}
    \includegraphics[width=\textwidth]{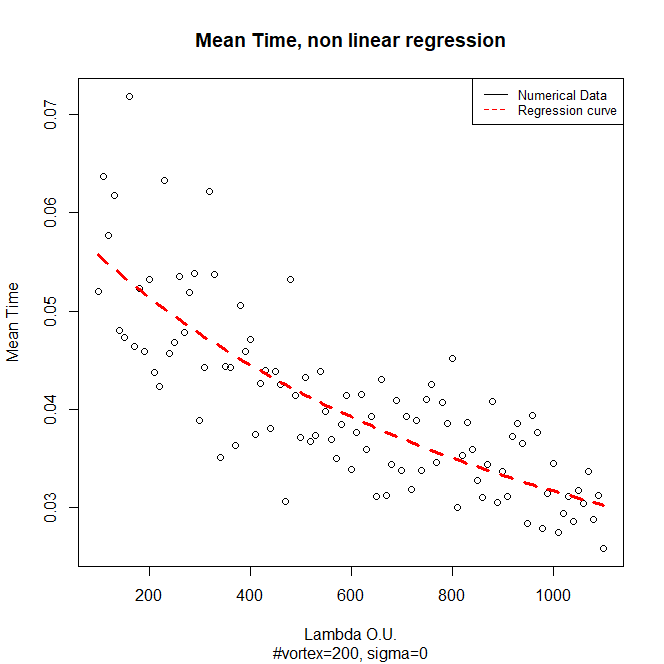}
  \end{minipage}
  \hfill
  \begin{minipage}[b]{0.45\textwidth}
    \includegraphics[width=\textwidth]{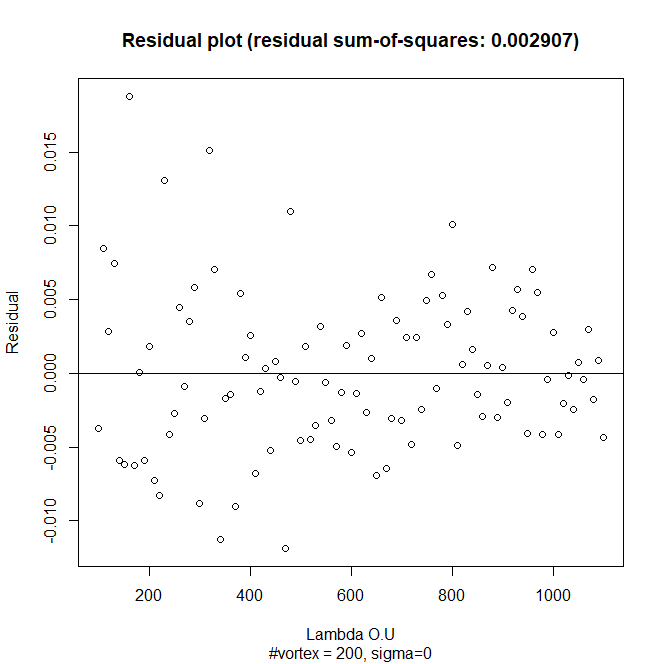}
  \end{minipage}
  \caption{Left:Non-linear Regression for moving $\lambda$ and fixed vortex. Expected decay as $\sim \lambda^{-1}$; Right: Residual plot. There is no presence of structure in the residual.}
\end{figure}

\printbibliography 

\end{document}